\documentclass[a4paper]{article}
\usepackage{enumerate}
\usepackage[english]{babel}
\usepackage[latin1]{inputenc}
\usepackage{graphicx}
\usepackage{color}
\usepackage{epsfig}
\usepackage{amsthm,amssymb,amsbsy,amsmath,amsfonts,amssymb,amscd,mathrsfs}
\usepackage{colortbl}
\usepackage{verbatim}
\usepackage{url}
\newtheorem{remark}{Remark}[section]
\newtheorem{proposition}{Proposition}[section]
\newtheorem{definition}{Definition}[section]
\newtheorem{lemma}{Lemma}[section]

\newtheorem{theorem}{Theorem}[section]
\newtheorem{corollary}{Corollary}[section]
\newtheorem{assumption}{Assumption}[section]
\newcommand{\vphi}{\varphi}

\newcommand{\eps}{\varepsilon}
\newcommand{\R}{\mathbb{R}}
\newcommand\be{\begin{equation}}
\newcommand\ee{\end{equation}}
\numberwithin{equation}{section}

\usepackage{xcolor}
\definecolor{linkcol}{rgb}{0.06,0.06,0.6}
\definecolor{citecol}{rgb}{0.6,0.04,0.04}
\definecolor{red}{rgb}{0.6,0.04,0.04}
\definecolor{blue}{rgb}{0.06,0.06,0.6}
\definecolor{green}{rgb}{0.06,0.6,0.06}
\definecolor{grey}{rgb}{0.85,0.85,0.85}

\usepackage[f]{esvect}          
\newcommand{\veps}{\varepsilon} %
\newcommand{\N}{\mathbb{N}}
\DeclareMathOperator{\GL}{GL}
\usepackage{tikz}
\usepackage{pgfkeys}
\usepackage{circuitikz}
\usetikzlibrary{arrows,shapes,backgrounds,patterns}
\usepackage{pifont}
\usepackage{tikzgraphicx}
\tikzstyle{every picture}+=[remember picture]
\tikzstyle{na} = [baseline=-.5ex]
\makeatletter
\newcommand{\gettikzxy}[3]{%
  \tikz@scan@one@point\pgfutil@firstofone#1\relax
  \edef#2{\the\pgf@x}%
  \edef#3{\the\pgf@y}%
}
\makeatother
\usepackage{mathtools}          
\newcommand{\arc}{\sigma}

\title{Tangency property and prior-saturation points in \\ minimal time problems in the plane}
\author{T. Bayen\footnote{Avignon Universit\'e, Laboratoire de Math\'ematiques d'Avignon (EA 2151) F-84018 {\tt\small terence.bayen@univ-avignon.fr}}
 ,
   O. Cots\footnote{Toulouse Univ., INP-ENSEEIHT, IRIT and CNRS, 2 rue Camichel, 31071 Toulouse, France  {\tt\small olivier.cots@irit.fr}}
}
\date{June 2, 2020}

\begin{document}
\maketitle
\begin{abstract}
In this paper, we consider minimal time problems governed by control-affine-systems in the plane, and 
we focus on the synthesis problem in presence of a singular locus that involves a saturation point for the singular control.  
After giving sufficient conditions on the data ensuring occurence of a prior-saturation point and a switching curve, we show 
that the bridge ({\it{i.e.}}, the optimal bang arc issued from the singular locus at this point) is tangent to the switching 
curve at the prior-saturation point. This property is proved using the Pontryagin Maximum Principle that also provides  
a set of non-linear equations that can be used to compute the prior-saturation  point. 
These issues are illustrated on a fed-batch model in bioprocesses 
and on a Magnetic Resonance Imaging (MRI) model for which minimal time syntheses for the point-to-point problem are discussed. 
\end{abstract}

\section{Introduction}
In this paper, we consider minimal time problems governed by single-input control-affine-systems in the plane
$$
\dot{x}(t)=f(x(t))+u(t)\, g(x(t)), \quad |u(t)|\leq 1, 
$$
where $f,g:\R^2\rightarrow \R^2$ are smooth vector fields. 
Syntheses for such problems have been investigated a lot in the literature (see, {\it{e.g.}}, \cite{bonnard,bosc,piccoli1,sussmann1,sussmann2,sussmann3}).  
In particular, an exhaustive description of the various encountered singularities can be found in \cite{bosc}, as well as an algorithm 
leading to the determination of optimal paths. It is worth mentioning that even though many techniques exist in this setting, 
the computation of an optimal feedback synthesis (global) remains in general difficult  because  of the occurence of geometric loci such as singular arcs, switching curves, cut-loci...
   
Our aim in this work is to focus on the notion of {\it{singular arc}} which appears in the synthesis when the switching function 
(the scalar product between the adjoint vector and the controlled vector field $g$) 
vanishes over a time interval. In the
context of control-affine-systems, singular arcs can generically be explicitly retrieved by solving non-degenerate linear equations 
(see, {\it{e.g.}}, \cite{kupka}).
Besides, in the two-dimensional case,  the corresponding singular control $u_s$ (which allows the associated trajectory to stay on the singular locus) 
can be expressed in feedback form 
$x\mapsto u_s[x]$. However, it may happen that $u_s$ becomes non admissible, {\it{i.e.}}, $x\mapsto u_s[x]$ takes values above the maximal value for the control (namely $1$ here). 
Such a situation naturally appears in several application models, see, {\it{e.g.}}, \cite{BHS,BMM2,ledje2,landfill}.  
In that case, we say that a {\it{saturation phenomenon}} occurs. The occurence of such a phenomenon implies the following (non-intuitive) property that,  
if a singular arc is optimal, then it should leave the singular locus at a so-called {\it{prior-saturation point}} before reaching the saturation point. 
This property has been studied in the literature in various situations such as 
for control-affine systems in dimension $2$ and $4$ (see, {\it{e.g.}}, \cite{schaettler,schaettler2,BDM95,BMM2} and references herein). 

Our main goal in this paper is to provide new qualitative properties on the minimum time synthesis in presence of a saturation point. 
More precisely, our objective is twofold:
\begin{itemize}
\item We  first give a set of conditions on the system that ensure  occurence of prior-saturation  showing that, under certain assumptions, the system leaves the singular arc at this point (before reaching the saturation point) with the maximal value for the control, see Proposition \ref{existence-presat}. This last arc is usually called {\it{bridge}} following the terminology as  in \cite{cots,cots2} (see also \cite{bonnard95,bonnard}). 
 \item Second, we  introduce a shooting function that allows an effective computation of the prior-saturation point. This mapping is used to 
 show our main result (Theorem \ref{thm-tangent}) which can be stated as follows: when the system exhibits a switching curve emanating from the prior-saturation point, then this curve is tangent to the bridge (in the cotangent bundle) at this point. 
\end{itemize}

The tangency property (in the state space) has been pointed out in several application models (see, {\it{e.g.}}, \cite{BMM2,cots}). To the best of our knowledge, 
this property has not been addressed previously in this general setting in the literature. It allows to better understand the construction of optimal paths 
locally at the prior-saturation point.  

The paper is structured as follows: in Section \ref{sec-general}, we recall classical expressions and properties 
of singular controls for control-affine-systems in the plane introducing the saturation phenomenon.  
In Section \ref{sec-prior}, we provide a set of  conditions involving the target set and the system ensuring the occurence of the prior-saturation phenomenon. 
In Section \ref{sec-tangent}, we show the tangency property between the switching curve emanating from a prior-saturation point and the bridge, 
and we describe how to compute the prior-saturation point thanks to a shooting function constructed via the Hamiltonian lifts of $f$ and $g$. 
Finally, we 
depict  this geometrical property in Section \ref{sec-numer} for a fed-batch model \cite{M99,BMM2} and MRI model \cite{cots,cots2}.   
This allows us to illustrate the notion of bridge in various contexts: first, when it connects a component of the singular locus to another one 
(see the MRI-model in Section \ref{sec-numer} and \cite{cots,cots2}), and then when it connects a component of 
a singular locus to an extended target set (see the fed-batch model in Section \ref{sec-numer} and \cite{M99,BMM2}). 

\section{Saturation phenomenon}{\label{sec-general}}
The purpose of this section is to recall some facts about  minimum time control problems in the plane that will allow us to introduce the saturation phenomenon. 
Throughout the paper, the standard inner product in $\R^2$ is
written $a\cdot b$ for $a,b\in \R^2$, and $a^\perp$ denotes the vector $a^\perp\coloneqq(-a_2,a_1)$ orthogonal to $a$.
The interior of a subset $S\subset \R^n$, $n\geq 2$, is denoted by $\mathrm{Int}(S)$.  

\subsection{Pontryagin's Principle}
We start by applying the classical optimality conditions provided by the Pontryagin Maximum Principle (PMP), see \cite{Pontry}. 
Let $f,g:\R^2 \rightarrow \R^2$  be two vector fields of class $C^\infty$, and consider the controlled dynamics:
\be{\label{sys1}}
\dot{x}(t)=f(x(t))+u(t)\, g(x(t)), 
\ee
with admissible controls in the set
$$
\mathcal{U}\coloneqq\{u:[0,+\infty) \rightarrow [-1,1]\; ; \; u \; \mathrm{meas}. \}. 
$$
Given an initial point $x_0\in \R^2$ and a non-empty closed subset $\mathcal{T}\subset \R^2$, we focus on the problem of driving \eqref{sys1}
in minimal time from $x_0$ to the target set $\mathcal{T}$:
\be{\label{OCP}}
\inf_{u \in \mathcal{U}} T_u \quad \mathrm{s.t.}\;\; x_u(T_u)\in \mathcal{T},
\ee 
where $x_u(\cdot)$ denotes the unique solution of \eqref{sys1} associated with the control $u$ such that $x_u(0)=x_0$, and $T_u\in [0,+\infty]$ is the first entry time of $x_u(\cdot)$ into the target set $\mathcal{T}$. 
We suppose hereafter that optimal trajectories exist\footnote{If the target can be reached from $x_0$ and if $f,g$ have linear growth, then \eqref{OCP} admits an optimal solution, thanks to  Filippov's Existence Theorem, see, {\it{e.g.}},  \cite{vinter}.} and we wish to apply the PMP  on \eqref{OCP}. The Hamiltonian associated with \eqref{OCP} is the function $H:\R^2\times \R^2\times \R \times \R \rightarrow \R$ defined as 
$$
H(x,p,p^0,u)\coloneqq p \cdot f(x)+u\, p \cdot g(x) +p^0.
$$
If $u$ is an optimal control and $x_u$ is the associated trajectory steering $x_0$ to the target set $\mathcal{T}$ in time $T_u \geq 0$, the following conditions are fulfilled:
\begin{itemize}
\item There exist $p^0 \leq 0$ and an absolutely continuous function $p:[0,T_u]\rightarrow \R^2$ satisfying the adjoint equation
\be{\label{adjoint}}
\dot{p}(t)=-\nabla_x H(x_u(t),p(t),p^0,u(t)) \quad  \mathrm{a.e.} \ t\in [0,T_u].
\ee
\item The pair $(p^0,p(\cdot))$ is non-zero. 
\item The optimal control $u$ satisfies the {\it{Hamiltonian condition}} 
\be{\label{PMP}}
u(t) \in \mathrm{argmax}_{\omega\in [-1,1]} H(x_u(t),p(t),p^0,\omega) \quad \mathrm{a.e.} \; t\in [0,T_u].
\ee
\item At the terminal time, the {\it{transversality condition}}\footnote{Here, $N_{\mathcal{T}}(x)$ stands for the (Mordukovitch) limiting normal cone 
to $\mathcal{T}$ at point $x\in \mathcal{T}$, see \cite{vinter}. It coincides with the normal cone in the sense of convex analysis when $\mathcal{T}$ is convex.} $p(T_u)\in -N_{\mathcal{T}}(x_u(T_u))$ is fulfilled. 
\end{itemize}
Recall that an extremal $(x_u(\cdot),p(\cdot),p^0,u(\cdot))$ satisfying \eqref{sys1} and \eqref{adjoint}-\eqref{PMP} 
is {\it{abnormal}} whenever $p^0=0$ and {\it{normal}} whenever $p^0\not=0$. In the latter case, we take $p^0=-1$ and the corresponding extremal is denoted by $(x_u(\cdot),p(\cdot),u(\cdot))$ and we shall then write $H(x,p,u)$ in place of $H(x,p,p^0,u)$. Since $T_u$ is free and \eqref{sys1} is autonomous, the Hamiltonian $H$ is zero along any extremal: for a.e.~$t\in [0,T_u]$, 
\be{\label{Hzero}}
H=p(t)\cdot f(x_u(t))+u(t) p(t)\cdot g(x_u(t))+p^0=0. 
\ee
 The {\it{switching function}} $\phi$ is defined as
\be
\label{eq:switching_function}
\phi(t)\coloneqq p(t) \cdot g(x_u(t)), \quad t\in [0,T_u], 
\ee
and it gives us (thanks to \eqref{PMP}) the following control law:
\be
\left\{
\begin{array}{lll}
\phi(t)>0 & \Rightarrow & u(t)=+1,\\
\phi(t)<0 & \Rightarrow & u(t)= -1.\\
\end{array}
\right.
\ee  
A {\it{switching time}} is an instant $t_c\in(0,T_u)$ such that the control $u$ is discontinuous at time $t_c$. We say that 
the corresponding extremal trajectory has a {\it{switching point}} at time $t_c$. 
Of particular interest is the case when there is a time interval $[t_1,t_2]$ such that the switching function vanishes over this interval, {\it{i.e.}}, 
$$
\phi(t)=p(t)\cdot g(x_u(t))=0, \quad t\in I.
$$
We then say that the extremal trajectory has a {\it{singular arc}} over $[t_1,t_2]$. 
Note that we shall suppose such an extremal to be normal, {\it{i.e.}}, $p^0\not=0$. Indeed, recall from 
\cite[Prop.~2 p.49]{bosc} that under generic conditions, abnormal extremals are bang-bang. 
By differentiating $\phi$ twice w.r.t.~$t$, one gets
$$
\dot{\phi}(t)=p(t)\cdot [f,g](x_u(t)),  \quad t\in [0,T_u], 
$$
where $[f,g](x)$ is the Lie bracket of $f$ and $g$ at point $x$, and 
$$
\ddot{\phi}(t)=p(t) \cdot [f,[f,g]](x_u(t))+u(t)\, p(t) \cdot [g,[f,g]](x_u(t)) \quad \mathrm{a.e.} \; t\in [0,T_u]. 
$$
The {\it{singular locus}} $\Delta_{SA}$ (in the state space) is defined as the (possibly empty) subset of $\R^2$ 
\be
\Delta_{SA}\coloneqq\{x\in \R^2 \; ; \; \mathrm{det}(g(x),[f,g](x))=0\}. 
\ee
For future reference, we set $\delta_{SA}(x)\coloneqq\mathrm{det}(g(x),[f,g](x))$ for  $x\in \R^2$. 
Note that if an extremal is singular over a time interval $[t_1,t_2]$, then one has $x_u(t)\in \Delta_{SA}$ for any $t\in [t_1,t_2]$ because 
$p(\cdot)$ must be non-zero and orthogonal to the vector space $\mathrm{span}\{g(x_u(t)),[f,g](x_u(t))\}$ over $[t_1,t_2]$.  
The {\it{singular control}} $u_s$ is then the value of the control for which the trajectory stays on the singular locus $\Delta_{SA}$. 
Supposing then that $\phi(t)=\dot{\phi}(t)=0$ over $[t_1,t_2]$ gives:
\be{\label{singular-control}}
u_s(t)\coloneqq-\frac{p(t) \cdot [f,[f,g]](x_u(t))}{p(t) \cdot[g,[f,g]](x_u(t))}, \quad t\in [0,T_u],
\ee
provided that $p(t) \cdot [g,[f,g]](x_u(t))$ is non zero for $t\in [t_1,t_2]$. 
This expression of the singular control does not guarantee that $u_s$ is admissible, that is, 
$u_s(t)\in [-1,1]$: 
\smallskip

\begin{itemize}
\item When we have $u_s(t)\in [-1,1]$, the point $x_u(t)$ is said 
{\it{hyperbolic}} if $p(t) \cdot[g,[f,g]](x_u(t))>0$, and {\it{elliptic}} if $p(t) \cdot[g,[f,g]](x_u(t))<0$ (see \cite{bonnard95,bonnard}). 
\item When we have $|u_s(t)|>1$ for some instant $t$, we say that a {\it{saturation phenomenon}} occurs and  
that the corresponding points of the singular locus are {\it{parabolic}} (see \cite{bonnard95,bonnard}). 
\end{itemize}
Our purpose in what follows is precisely to investigate properties 
of the synthesis of optimal paths when saturation occurs. To this end, we suppose 
in the rest of the paper that extremals are normal, {\it{i.e.}}, $p^0\not=0$ (we take hereafter $p^0=-1$). 
\subsection{Singular control and saturation phenomenon}
In this part, we derive classical expressions of the singular control in terms of feedback that will allow us to introduce saturation points 
(in terms of the data defining the system).   
The {\it{collinearity set}} associated with \eqref{sys1} is the (possibly empty) subset of $\R^2$ defined as
\be
\Delta_{0}\coloneqq\{x\in \R^2 \; ; \; \mathrm{det}(f(x),g(x))=0\}. 
\ee
Define two functions $\delta_0,\psi:\R^2\rightarrow \R$ as 
$\delta_0(x)\coloneqq\mathrm{det}(f(x),g(x))$, $x\in \R^2$, and  
\be{\label{feedback-singulier}}
\psi(x)\coloneqq-\frac{\mathrm{det}(g(x),[f,[f,g]](x))}{\mathrm{det}(g(x),[g,[f,g]](x))}, \quad x\in \R^2. 
\ee
The singular control can be then expressed as follows. 
\begin{lemma}
Suppose that $\Delta_{SA}\not=\emptyset$, that $x\mapsto \mathrm{det}(g(x),[g,[f,g]](x))$ is non-zero over $\Delta_{SA}$, and consider 
a singular arc defined over an interval $[t_1,t_2]$. Then, one has:
\be{\label{feedback}}
u_s(t)=\psi(x(t)), \quad t\in [t_1,t_2], 
\ee
where $x(\cdot)$ is the corresponding singular trajectory such that $x(t)\in \Delta_{SA}$ for $t\in [t_1,t_2]$. 
\end{lemma}
\begin{proof} The proof is classical and combines 
\eqref{singular-control} together with the equalities 
$$
\begin{array}{c}
-\delta_0(x(t))\, p(t)\cdot [f,[f,g]](x(t))=\mathrm{det}(g(x(t)),[f,[f,g]](x(t))),  
\vspace{0.1cm} \\
-\delta_0(x(t))\, p(t)\cdot [g,[f,g]](x(t))=\mathrm{det}(g(x(t)),[g,[f,g]](x(t))), 
\end{array}
$$
see, {\it{e.g.}}, \cite[Lemma 10]{bosc}. 
\end{proof}
\begin{remark}
{\it{Steady-state singular points}} are defined as the points $x^\star\in \Delta_{SA}\cap\Delta_0$ such that  
$g(x^\star)\not=0$, see {\rm{\cite{bosc,BRS14}}} (if $\Delta_{SA}\cap\Delta_0\not=\emptyset$). 
Such points are equilibria of \eqref{sys1} with $u=\psi(x)$. 
A singular arc defined over a time interval $[t_1,t_2]$ does 
not contain such a point because $f(x(t))$ and $g(x(t))$ must be linearly independent over $[t_1,t_2]$. 
But, it can contain points $x^\star \in \Delta_{SA}\cap \Delta_{0}$ such that $g(x^\star)=0$. 
\end{remark}
In the sequel, we consider a parametrization of the set $\Delta_{SA}$ as follows. This will be useful to introduce the notion of saturation point.   
\begin{lemma}{\label{lem-para}}
Suppose that $\Delta_{SA}$ is non-empty and that $x\mapsto \mathrm{det}(g(x),[g,[f,g]](x))$ is non-zero over $\Delta_{SA}$. 
Then, 
{near each point $x_0\in \Delta_{SA}\backslash \Delta_0$}, the set $\Delta_{SA}\backslash \Delta_0$ 
can be locally parametrized by a one-to-one parametrization $\zeta:J\rightarrow \Delta_{SA}$, 
$\tau\mapsto \zeta(\tau)$ of class $C^1$, where $J$ is an interval of $\R$. 
\end{lemma}
\begin{proof}
For $x\notin \Delta_0$, 
one has $\mathrm{span}\{f(x),g(x)\}=\R^2$, hence, there exist 
$\alpha(x), \beta(x)\in \R$ such that
\be{\label{decompo1}}
[f,g](x)=\alpha(x) f(x)+\beta(x)g(x). 
\ee
By taking the determinant between $g(x)$ and $[f,g](x)$, and then between $f(x)$ and $[f,g](x)$, we find that for $x\notin \Delta_0$, 
$$
\alpha(x)=-\frac{\mathrm{det}(g(x),[f,g](x))}{\delta_0(x)}  \; \; \mathrm{and} \;  \; 
\beta(x)=\frac{\mathrm{det}(f(x),[f,g](x))}{\delta_0(x)}. 
$$
By computing $[g,[f,g]](x)$ thanks to \eqref{decompo1}, we get
$$
\mathrm{det}(g(x),[g,[f,g]](x))=-\delta_0(x) \nabla \alpha(x)\cdot g(x), \quad x\notin \Delta_0. 
$$
Since $x\mapsto \mathrm{det}(g(x),[g,[f,g]](x))$ is non-zero over $\Delta_{SA}$, the preceding equality implies that 
the scalar product $\nabla \alpha(x)\cdot g(x)$ is non-zero. 
On the other hand, $\Delta_{SA}$ 
is defined by the implicit equation $\delta_{SA}(x)=0$. 
Next, observe that 
\begin{equation}{\label{alphadeltaSA}}
x\notin \Delta_0 \; \Rightarrow  \delta_{SA}(x)=-\alpha(x)\delta_0(x).
\end{equation}
By taking the derivative, we find that for $x\notin\Delta_0$, one has
$$
\nabla \delta_{SA}(x)=-\delta_0(x) \nabla \alpha(x)-\alpha(x) \nabla \delta_0(x).
$$ 
Therefore, for $x\in\Delta_{SA} \backslash \Delta_0$, we obtain 
$\nabla \delta_{SA}(x)=-\delta_0(x) \nabla \alpha(x)$. 
We can  then conclude that for any point $x_0\in \Delta_{SA} \backslash \Delta_0$,  
the partial derivative $\partial_1 \alpha(x_0)$ (w.r.t.~$x_1$) or $\partial_2 \alpha(x_0)$ (w.r.t.~$x_2$) 
is non-zero. We are then in a position to apply the implicit function theorem to $\delta_{SA}$ {locally at each point $x_0\in \Delta_{SA}\backslash \Delta_0$}, which  implies the desired property. 
\end{proof}
When $\Delta_{SA} \cap \Delta_0$ is non-empty, the set 
$\Delta_{SA} \backslash \Delta_0$ can be then partitioned into (maximal) connected one dimensional submanifolds of $\Delta_{SA}$, according to the previous lemma.
Hence, we write this set as
$$
\Delta_{SA} \backslash \Delta_0=\bigcup_{k \in K} \gamma_k,
$$
where $K$ is an index set. Hereafter, to shorten, the terminology ``component'' refers to  a (maximal) connected one-dimensional manifold $\gamma_k$ with $k \in K$. 
Under the assumptions of Lemma \ref{lem-para}, given a component $\gamma_k$ of $\Delta_{SA}$, there is a parametrization $\zeta$ such that
$$
\gamma_k\coloneqq\{\zeta(\tau) \; ; \; \tau \in J\},
$$  
where $\zeta:J \rightarrow \R^2$ is $C^1$-mapping (injective) and $J$ is an interval. 
\begin{definition} 
A point $x^*\coloneqq\zeta(\tau^*)$ with $\tau^*\in \mathrm{Int}(J)$ 
is called {\it{saturation point}} if $|\psi(x^*)|=1$ and: 
\begin{itemize}
\item either, $|\psi(\zeta(\tau))|<1$ for any $\tau\in J$ such that $\tau<\tau^*$, and $|\psi(\zeta(\tau))|>1$ for any $\tau \in J$ such that $\tau >\tau^*$, 
\item or $|\psi(\zeta(\tau))|>1$ for any $\tau\in J$ such that $\tau<\tau^*$, and $|\psi(\zeta(\tau))|<1$ for any $\tau \in J$ such that $\tau >\tau^*$.
\end{itemize} 
\end{definition}
Note that, in this paper, we shall mainly consider those saturation points such that  
$\psi(x^*)=1$ where $\psi$ is increasing over $\gamma_k$.


Our next aim is to study the optimality 
of singular arcs in presence of a saturation point. 
\section{Existence of a prior-saturation point}{\label{sec-prior}}
In this section, we show that a prior-saturation phenomenon can occur whenever the system exhibits a saturation point. 
We start by introducing our main assumptions. 
\begin{assumption}{\label{main-hyp}}
The system \eqref{sys1} satisfies the following hypotheses: 
\begin{itemize} 
\item[$\mathrm{(i)}$] One has $\Delta_0=\emptyset$ and $\delta_0(x)<0$ for all $x\in \R^2$ {so that  
the singular locus $\Delta_{SA}$ is written $\Delta_{SA}\coloneqq\zeta(J)$ where $J\subset \R$ is an interval and $\zeta:J\rightarrow \Delta_{SA}$ is a $C^1$-mapping.}
\item[$\mathrm{(ii)}$] The set $\Delta_{SA}$ is non-empty, simply connected, and has exactly one saturation point {$x^*=\zeta(\tau^*)$ with $\psi(x^*)=1$ and $\tau^*\in J$.} Moreover, $\tau \mapsto \psi(\zeta(\tau))$ is increasing over $J$. 
\item[$\mathrm{(iii)}$] Along the singular locus, the strict (generalized) Legendre-Clebsch optimality condition is satisfied, that is, any singular extremal 
$(x_u(\cdot),p(\cdot),u(\cdot))$ defined over $[t_1,t_2]$ satisfies:
\be{\label{LC-strict}}
\frac{\partial }{\partial u} \frac{d^2}{dt^2} \frac{\partial H}{\partial u}(x_u(t),p(t),u(t)) > 0, \quad \forall t\in [t_1,t_2]. 
\ee
\item[$\mathrm{(iv)}$] If $\Gamma_{-}$ is the forward semi-orbit of \eqref{sys1} with $u=-1$ with the initial condition $x^*$ at time $0$, then
\be{\label{intersec0}}
\mathcal{T}\cap \Gamma_{-}=\emptyset.
\ee
\item[$\mathrm{(v)}$] The target $\mathcal{T}$ is reachable from every point $x_0\in \R^2$. 
\end{itemize}
\end{assumption}

\begin{remark}{\label{rem0}} {\rm{(i)}} The hypothesis $\Delta_0=\emptyset$ could be weakened considering in place of $\R^2$ an optimally\footnote{By optimally invariant subset, we mean a subset $\Omega \subset \R^2$ such that for any initial condition $x_0\in \Omega$, an optimal trajectory stays in $\Omega$.} invariant subset $\Omega$ (by the dynamics) containing only one component $\gamma_k$ and the target set. In addition, we would require the condition $\delta_0<0$ in $\Omega$. For simplicity, we supposed here that $\Omega=\R^2$. 
\\
{\rm{(ii)}} By the previous computations, we can observe that  
\eqref{LC-strict} is equivalent to
$$
\forall x\in \Delta_{SA}, \; |\psi(x)|\leq 1 \; \Rightarrow \; \mathrm{det}(g(x),[g,[f,g]](x))>0.  
$$
Recall that, under the strict Legendre-Clebsch condition, the singular arc is a {\it{turnpike}}, {\it{i.e.}}, it is time-minimizing in every neighborhood of a hyperbolic point of $\Delta_{SA}$, {\rm{\cite{bonnard95}}}. This property can be retrieved by the clock form argument {\rm{\cite{Hermes}}}.   
\\
(iii) In general, the saturation point $x^*$ is not unique (this depends on the case study). 
Whereas in the two examples of Section \ref{sec-numer} it is unique, let us mention {\rm{\cite{BD2020}}} in which several saturation points as well as bridges occur in the context of chemical reactors. 
\end{remark}

Under Assumption \ref{main-hyp},  $\Delta_{SA}$ 
partitions the state space into two simply connected (open) subsets $\Delta_{SA}^\pm$:
$$
\begin{array}{c}
\Delta_{SA}^+\coloneqq\{x\in \R^2 \; ; \; \mathrm{det}(g(x),[f,g](x))>0\}, \\
\Delta_{SA}^-\coloneqq\{x\in \R^2 \; ; \; \mathrm{det}(g(x),[f,g](x))<0\}.
\end{array}
$$
Given a normal extremal $(x_u(\cdot),p(\cdot),u(\cdot))$, the function
$$
t\mapsto \gamma_u(t)\coloneqq\beta(x_u(t))-\alpha(x_u(t))u(t), \quad t\in [0,T_u],
$$  
is also well-defined since $\Delta_0=\emptyset$.

\begin{lemma}
Along a normal extremal $(x_u(\cdot),p(\cdot),u(\cdot))$, the switching function $\phi$ satisfies the ODE
\be{\label{ode-phi}}
\dot{\phi}(t)=\gamma_u(t)\phi(t)+\alpha(x_u(t)) \quad \mathrm{a.e.} \; t\in [0,T_u].
\ee
\end{lemma}
\begin{proof} The proof follows using the expression of $\dot{\phi}$ and the fact that the Hamiltonian $H$ is constant equal to zero. 
\end{proof}

The next proposition 
shows that an extremal trajectory containing a singular arc until the point $x^\star$ is not optimal.

\begin{proposition}{\label{existence-presat}} 
Suppose that Assumption \ref{main-hyp} holds true, and consider an optimal trajectory steering $x_0$ to the target 
$\mathcal{T}$ in time $T_u$. 
Then, the corresponding extremal $(x_u(\cdot),p(\cdot),u(\cdot))$ does not contain a singular arc defined over a time interval $[t_1,t_2]$ such that 
$x_u(t_2)=x^*$.
\end{proposition}
\begin{proof}
Suppose by contradiction that there is a time interval $[t_1,t_2]$ such that the trajectory is singular over $[t_1,t_2]$  with
$x_u(t_2)=x^*$. First, note that at time $t_2$, the vector $f(x_u(t_2))+g(x_u(t_2))$ is tangent to $\Delta_{SA}$. Indeed, since $x_u(t)\in\Delta_{SA}$ for $t\in [t_1, t_2]$, 
the vector $f(x_u(t))+\psi(x_u(t))g(x_u(t))$ is tangent to $\Delta_{SA}$ at every 
time $t\in [t_1,t_2]$. The result then follows because at time $t_2$, one has  
$\psi(x_u(t_2))=1$. 
 

From lemma \ref{lem-para}, observe that for $x\in \Delta_{SA}$, one has: 
\be{\label{tmp-signe}}
\begin{array}{rl}
\mathrm{det}(g(x),[g,[f,g]](x))&=-\delta_0(x)\, \nabla \alpha(x)\cdot g(x), \\ 
\mathrm{det}(g(x),[f,[f,g]](x))&=-\delta_0(x)\, \nabla \alpha(x)\cdot f(x).
\end{array}
\ee
We deduce the following equalities:
\begin{align}{\label{exclusion-tmp1}}
&\nabla \delta_{SA}(x)\cdot (f(x)+g(x))=\mathrm{det}(g(x),[g,[f,g]](x))\, (1-\psi(x)),\\
&\nabla \delta_{SA}(x)\cdot (f(x)-g(x))=\mathrm{det}(g(x),[g,[f,g]](x))\, (-1-\psi(x)).\notag
\end{align}
 Going back to the optimal trajectory, under the strict Legendre-Clebsch condition\footnote{Following for instance \cite[Section 2.8.4]{schaettler2}, chattering occurs for singular arcs of higher order (at least $2$) for which the Legendre-Clebsch condition is not fulfilled.}, there is 
no chattering phenomenon for the optimal control at time $t_2$. 
This excludes the trajectory to have an infinite number of switching times in a right neighborhood of $t=t_2$. Because the singular arc becomes non admissible, the trajectory is also not singular in a right neighborhood of $t_2$.

There remain now two possibilities for the trajectory $x_u(\cdot)$:   
 \\
- either one has $u=+1$ over some time interval $[t_2,t_2+\eps)$ with $\eps>0$, \\
- or one has $u=-1$ over some time interval $[t_2,t_2+\eps)$ with $\eps>0$. 

Suppose first that $u=+1$ over $[t_2,t_2+\eps)$. {Thanks to Assumption \ref{main-hyp} (ii), one has $\psi(\zeta(\tau))>1$ for every $\tau>\tau^*$ and Remark \ref{rem0} (ii) 
implies that
$$
\mathrm{det}(g(\zeta(\tau)),[g,[f,g]](\zeta(\tau)))>0. 
$$
From the expression of $\psi$ given by \eqref{feedback-singulier}, we deduce the inequalities
$$
\mathrm{det}(g(\zeta(\tau)),[g,[f,g]](\zeta(\tau)))>0 \; \; \mathrm{and}\; \; -\mathrm{det}(f(\zeta(\tau)),[g,[f,g]](\zeta(\tau)))>0,
$$
for every $\tau>\tau^*$. Taking into account \eqref{exclusion-tmp1}, we obtain that 
$$
\tau>\tau^* \; \Rightarrow \; \nabla \delta_{SA}(\zeta(\tau))\cdot (f(\zeta(\tau))+g(\zeta(\tau)))<0.
$$  
One can then infer that for $\tau>\tau^*$, each vector $f(\zeta(\tau))+g(\zeta(\tau))$ at $\zeta(\tau)\in \Delta_{SA}$ points into $\Delta_{SA}^-$. }
{It follows that 
the corresponding trajectory necessarily enters into $\Delta_{SA}^-$ implying (recall \eqref{alphadeltaSA})
$$
\mathrm{det}(g(x_u(t)),[f,g](x_u(t)))=-\delta_0(x_u(t))\alpha(x_u(t))<0
$$ 
for $t\in (t_2,t_2+\eps]$. From \eqref{ode-phi}, since $\phi(t_2)=\dot{\phi}(t_2)=0$, we obtain that $\phi(t)<0$ 
for  $t\in (t_2,t_2+\eps]$. From Pontryagin's Principle, $\phi<0$ implies that $u(t)=-1$. We thus have a contradiction. }

It follows that the optimal trajectory necessarily satisfies $u=-1$ in some time interval $(t_2,t_2+\eps]$, and thus it enters into the set $\Delta_{SA}^-$ because one has $\nabla \delta_{SA}(x^*)\cdot (f(x^*)-g(x^*))<0$.

From Assumption \ref{main-hyp}, the forward semi-orbit with $u=-1$ starting from $x^*$ does not reach the target set. Hence, $x_u(\cdot)$ must have a switching point 
to $u=+1$ in $\Delta_{SA}^-$ or it must reach $\Delta_{SA}$ with the control $u=-1$.   
We see from \eqref{ode-phi} that the first case is not possible because at a switching time $t_c$ such that $x_u(t_c)\in \Delta_{SA}^-$, we would have 
$\dot{\phi}(t_c)\geq 0$ in contradiction with $\alpha(x_u(t_c))<0$. 

Suppose now that $x_u(\cdot)$ reaches $\Delta_{SA}$ at some point $x\coloneqq\zeta(\tau)$ with $\tau<\tau^*$. 
Then, we obtain 
$\nabla \delta_{SA}(x)\cdot (f(x)-g(x))<0$ since $\psi(x)>-1$. But, 
as $x_u(\cdot)$ reaches $\Delta_{SA}$ with $u=-1$ at point $x$, the trajectory enters into the set $\Delta_{SA}\cup\Delta_{SA}^+$ and we must have 
$\nabla \delta_{SA}(x)\cdot (f(x)-g(x)) \geq 0$ ($\nabla \delta_{SA}(x)$ is collinear to the outward normal vector to $\Delta_{SA}$ at point $x$). This gives a contradiction. 
In the same way, the trajectory cannot reach 
a point $x\in \Delta_{SA}$ such that $x=\zeta(\tau)$ with $\tau>\tau^*$. 

We can conclude that for any time $t\geq t_2$, one has $u(t)=-1$, but then, the optimal trajectory cannot reach the target set which is a contradiction 
(Assumption \ref{main-hyp} (iv)). This concludes the proof. 
\end{proof}

As an example, if $x_0\coloneqq\zeta(\tau_0)$ belongs to the singular locus with $\tau_0<\tau^*$, and if an optimal trajectory starting from $x_0$ contains a singular arc, then the trajectory should leave the singular locus before reaching $x^*$. Let us insist on the fact that 
this property of leaving the singular locus before reaching $x^*$ relies on the fact that the optimal trajectory should contain a singular arc.  
In the fed-batch model presented in Section \ref{fb-sec}, this property can be easily verified (see \cite{BMM2}).  

We now introduce the following definition (in line with \cite{ledje2,schaettler,schaettler2}).  
Hereafter, the notation $\mathcal{S}_{[\tau'_0,\tau_0]}$ 
 denotes the singular arc comprised between the points $\zeta(\tau'_0)$ and 
$\zeta(\tau_0)$ with $\tau'_0\leq \tau_0<\tau^*$.

\begin{definition}{\label{defi-prior}}
Let $\tau_0<\tau^*$. 
A point $x_e\coloneqq\zeta(\tau_e)\in \Delta_{SA}$ with $\tau_0< \tau_e<\tau^*$ is called a {\it{prior-saturation point}} if 
the singular arc $\mathcal{S}_{[\tau_0,\tau]}$ ceases to be optimal for  $\tau\geq \tau_e$, that is,
$$
\tau_e=\sup \{\tau \in J \; \; ; \; \; \mathcal{S}_{[\tau_0,\tau]} \;  \mathrm{is} \;  \mathrm{optimal}\}.
$$
\end{definition}
This definition makes sense only for initial conditions $\zeta(\tau_0)$ with $\tau_0<\tau^*$ because for $\tau_0\geq \tau^*$, 
optimal controls are not singular (since the singular control is non-admissible). Note that $x_e$ a priori depends on the initial condition $\zeta(\tau_0)$ on the singular locus (hence, we shall write $x_e=x_e(\tau_0)$ hereafter if necessary).
\begin{proposition}{\label{prior-vs-init}} Suppose that Assumption \ref{main-hyp} holds true and  
that there are $\tau_1,\tau_2\in J$ with $\tau_1<\tau_2<\tau^*$ 
such that any optimal trajectory starting from $\zeta(\tau_0)$ with $\tau_0\in [\tau_1,\tau_2)$ contains a singular 
arc $\mathcal{S}_{[\tau_0,\tau_2]}$. Then, one has $x_e(\tau_0)=x_e(\tau_2)$ for every $\tau_0\in [\tau_1,\tau_2]$. 
\end{proposition}
\begin{proof}
Consider an optimal trajectory $\gamma$ 
starting from $\zeta(\tau_0)\in \Delta_{SA}$ 
with $\tau_0\in [\tau_1,\tau_2]$. The hypothesis of the proposition implies that this trajectory passes though the point $\zeta(\tau_2)$. In addition, from the initial condition $\zeta(\tau_2)$, 
an optimal trajectory necessarily leaves the singular locus at a prior saturation point denoted by 
$\bar x_e$. It follows that the optimal trajectory starting from $\zeta(\tau_0)$ also leaves $\Delta_{SA}$ at the same point $\bar x_e$. Otherwise, $\gamma$ would leave 
$\Delta_{SA}$ at a point $x_e\not=\bar x_e$. Since 
$\gamma$ is optimal, the optimal trajectory from  $\zeta(\tau_2)$ would also leave $\Delta_{SA}$ at $x_e$ implying a contradiction. 
\end{proof}
This property implies  that for every initial conditions  $x_0\coloneqq\zeta(\tau_0)\in \Delta_{SA}$ such that $\tau_0\in [\tau_1,\tau_2]$, then 
the corresponding optimal path has a singular arc until the point $x_e(\tau_2)$ and a switching point at this point. 
\begin{remark} 
Such a situation (in particular the existence of an interval $[\tau_1,\tau_2]$ as in Proposition \ref{prior-vs-init}) is encountered in the fed-batch model presented in Section \ref{fb-sec} (see also {\rm \cite{BMM2}}).  Note also that for initial conditions 
$\zeta(\tau_0)\in \Delta_{SA}$ with $\tau_0<\tau^*$, $\tau_0$ close to $\tau^*$, we can expect that $\tau_e(\tau_0)=\tau_0$, {\it{i.e.}}, the singular arc (altough being admissible) is no longer optimal.
\end{remark}
%

\begin{remark}
In addition to Assumption \ref{main-hyp} (in particular \eqref{intersec0}), if we suppose that $\mathcal{T}$ is not reachable with the constant control $u=-1$ 
from those points of $\Delta_{SA}$ located between $x_e$ and $x^*$ ({\it{i.e.}} corresponding to $\tau\in [\tau_e,\tau^*]$),  
then the maximal value for the control $u=+1$ is locally optimal from the prior-saturation point $x_e$. 
In other words, the {\it{bridge}} (the last arc leaving $\Delta_{SA}$) corresponds to $u=+1$. This can be proved by using similar arguments as 
for proving Proposition \ref{prior-vs-init}. 
Since the singular arc is a turnpike, this additional hypothesis also implies the existence of a switching curve emanating from $x_e$. 
Our next aim is precisely to investigate more into details geometric properties of optimal paths at the point $x_e$. 
\end{remark}


\section{Tangency property and prior-saturation phenomenon}{\label{sec-tangent}}
The aim of this section is to prove the tangency property as stated in Theorem \ref{thm-tangent}.  

\subsection{Introduction to prior-saturation lift and tangency property}\label{sec:prior_lift_intro}

In this section, we first introduce the concept of {\it{prior-saturation lift}} and discuss its local uniqueness.
We also provide a set of nonlinear equations allowing to compute prior-saturation lifts given by the PMP. 
We end this section with an introduction to the tangency property on an example.


\begin{definition}
Let $x_e$ be a prior-saturation point. 
Any point $z_e$ in the cotangent space at $x_e$ is called a \emph{prior-saturation lift} of $x_e$.
\end{definition}

To introduce the computation of prior-saturation lifts given by the PMP, let us start with an example.
Consider a target set $\mathcal{T} \coloneqq \{x_f\}$, $x_f \in \R^2$,
with an optimal trajectory of the form $\arc_- \arc_s \arc_+$, where $\arc_-$, $\arc_+$ and $\arc_s$ are arcs, respectively,
with control $u=-1$, $u=+1$ and $u = u_s$, where $u_s$ is the singular control.
Assume
that the switching point between the singular arc $\arc_s$ (supposed to be non-empty)
and the positive bang arc $\arc_+$ is a prior-saturation point.
The PMP\footnote{Since $\mathcal{T} \coloneqq \{x_f\}$ is a point, there is no transversality condition at the terminal time.}  gives necessary optimality conditions satisfied by this extremal trajectory that we can write as a system of nonlinear equations,
the so-called {\it{shooting equations}}.
Actually, the shooting equations, that we present hereinafter, will give us
the initial adjoint vector and the switching times together with the switching points.
From this, we can retrieve the extremal trajectory.
Before defining this set of equations, we need to introduce some notation.
%
We define the Hamiltonian lifts associated with $f$ and $g$ as $$H_f(z) \coloneqq p \cdot f(x) \; ; \; H_g(z) \coloneqq p \cdot g(x),$$ where $z \coloneqq (x,p)$
belongs to the cotangent bundle. 
All the other Hamiltonian lifts in the rest of the paper are defined like this.
Define also the Hamiltonians $H_\pm \coloneqq H_f \pm H_g$ and $H_s \coloneqq H_f + u_s H_g$, where $u_s$ is viewed here as a function of $z$:
\begin{equation}
    u_s(z) \coloneqq -\frac{p \cdot [f,[f,g]](x)}{p \cdot[g,[f,g]](x)} = -\frac{H_{[f,[f,g]]}(z)}{H_{[g,[f,g]]}(z)}.
    \label{eq:singular_control_z}
\end{equation}
For any Hamiltonian $H$ we define the {\it{Hamiltonian system}} $\vv{H} \coloneqq (\partial_p H, -\partial_x H)$, and finally, we introduce the
\emph{exponential mapping} 
$\exp(t\vphi)(z_0)$ as the solution at time $t$ of the differential equation
$\dot{z}(s) = \vphi(z(s))$ with initial condition  $z(0) = z_0$, where $\vphi$ is supposed to be smooth.
The {shooting equations} are then given by 
\[
    S(y) = 0, \quad y \coloneqq (p_0, t_1, t_2, t_f, z_1, z_2) \in \R^{n+3+(2n)\times 2}, \quad n\coloneqq2,
\]
where the \emph{shooting function} is defined by
\begin{equation}
      S(y) \coloneqq 
    \begin{pmatrix}
        H_g(z_1) \\
        H_{[f,g]}(z_1) \\
        H_+(\exp( (t_f-t_2) \vv{H_+})(z_2)) + p^0 \\
        \pi(\exp( (t_f-t_2) \vv{H_+})(z_2) ) - x_f \\
        z_1 - \exp( t_1 \vv{H_-})(x_0, p_0) \\
        z_2 - \exp( (t_2-t_1) \vv{H_s})(z_1)
    \end{pmatrix},  
    \label{eq:shooting_equation_ex}
\end{equation}
where $\pi(x,p)\coloneqq x$, and where $x_0\in\R^2$ is given and $p^0 = -1$ in the normal case.
The two first equations mean that the trajectory is entering the singular locus at $z_1$.
Hence, the second arc is a singular arc.
The third equation takes into account the free terminal time.
It could be replaced by $H_-(x_0, p_0) + p^0 = 0$ since the maximized Hamiltonian is constant along the extremal.
The fourth equation implies that the last bang arc reaches the target $\mathcal{T} = \{x_f\}$ at the final time $t_f$,
and the last two equations are the so-called {\it{matching conditions}} 
which give the switching points. 
Let consider a solution $y^* \coloneqq (p_0, t_1, t_2, t_f, z_1, z_2)$
to $S(y)=0$ associated with the optimal trajectory described before. Then, the point
$\pi(z_2)$ is a prior-saturation point, and so, $z_2$ is a prior-saturation lift.

Let us now discuss the uniqueness of the prior-saturation lift, considering for instance, a smooth and local one-parameter family of initial
conditions $x_0(\alpha)$, $\alpha \in (-\veps, \veps)$, $\veps > 0$, in relation with the construction of optimal syntheses (see
section \ref{sec-numer}) and in relation with Proposition \ref{prior-vs-init}. 
Let us denote by 
$y^*(\alpha) \coloneqq (p_0(\alpha), t_1(\alpha), t_2(\alpha), t_f(\alpha), z_1(\alpha), z_2(\alpha))$ a family of solutions to the equation $S(y, \alpha)=0$
for $\alpha \in (-\veps, \veps)$, where $S(\cdot,\alpha)$ is defined as 
\eqref{eq:shooting_equation_ex} but with the initial condition $x_0(\alpha)$ in place of
$x_0$. Let us assume that for any $\alpha\in (-\eps,\eps)$, the corresponding trajectory is
an optimal trajectory of the form $\arc_- \arc_s \arc_+$.
%
%
%
In addition, suppose that the lengths $t_1$, $t_2-t_1$ and $t_f-t_2$ are positive,
that is, each arc is defined on a time interval of positive length. 
In this setting, for any $\alpha$, we have $x_e \coloneqq \pi(z_2(0)) = \pi(z_2(\alpha))$, that is, the prior-saturation point $x_e$ is locally unique.
This is related to Proposition \ref{prior-vs-init} and illustrated on Figure~\ref{fig:bsb_sols}.  
Besides, under some other assumptions, see Proposition \ref{prop:prior_lift_unique}
and remark \ref{rmk:assumptions_lift}, the prior-saturation lift $z_e \coloneqq z_2(0)$ 
is also locally unique and so in this case we have also $z_2(0) = z_2(\alpha)$ for any 
$\alpha \in (-\veps,\veps)$. 

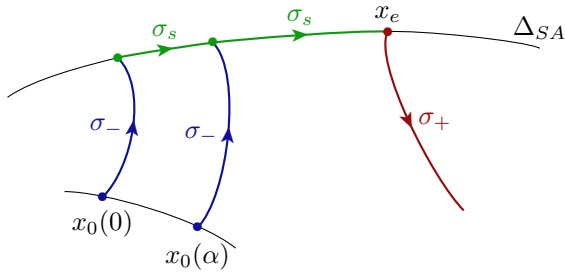
\begin{figure}[ht!]

    \begin{center}
    \begin{tikzpicture}[scale=0.05]

    \coordinate (A) at (-20,0);
    \coordinate (B) at (120,13);
    \coordinate (C) at (-5,-25);
    \coordinate (D) at (40,-40);
    \coordinate (x6) at (100,-30);

    
    \coordinate (x1) at (5,-26.5);
    \coordinate (x2) at (30,-34.5);
    \coordinate (x3) at (9,10.5);
    \coordinate (x4) at (34,14.5);
    \coordinate (x5) at (80,17.4);
    
    \draw [thin] (C) .. controls +(10,0) and +(-10,8) .. (D);
    
    \draw [thin, black] (A) .. controls +(7,5) and +(-8,-2) .. (x3);
    \draw [thin, black] (x5) .. controls +(8,0.5) and +(-4,2) .. (B);

    \draw [thick, green] (x3) .. controls +(4,1) and +(-4,-0.5) .. (x4) node[midway, sloped] {\scriptsize \ding{228}}
    node[midway, above]{$\textcolor{green}{\arc_s}$};
    
    \draw [thick, green] (x4) .. controls +(8,1) and +(-8,0.5) .. (x5) node[midway, sloped] {\scriptsize \ding{228}}
    node[midway, above]{$\textcolor{green}{\arc_s}$};
    
    \draw [thick, blue] (x1) .. controls +(8,8) and +(8,-8) .. (x3) node[midway, sloped] {\scriptsize \ding{228}}
    node[midway, left]{$\textcolor{blue}{\arc_-}$}; 
    
    \draw [thick, blue] (x2) .. controls +(8,8) and +(8,-8) .. (x4) node[midway, sloped] {\scriptsize \ding{228}}
    node[midway, left]{$\textcolor{blue}{\arc_-}$};
    
    \draw [thick, red] (x5) .. controls +(-4,-8) and +(-8,8) .. (x6) node[midway, sloped] {\scriptsize \ding{228}}
    node[midway, right]{$\textcolor{red}{\arc_+}$};
    
    \draw [above] (x5) node{$x_{e}$};
    \draw [above] (B)  node{$\Delta_{SA}$};
    \draw [below] ([yshift=-3em] x1) node{$x_0(0)$};
    \draw [below] ([yshift=-5em] x2) node{$x_0(\alpha)$};
    
     \draw [blue] (x1) node{\footnotesize $\bullet$}; 
    \draw [blue] (x2) node{\footnotesize $\bullet$}; 
    \draw [green] (x3) node{\footnotesize $\bullet$}; 
    \draw [green] (x4) node{\footnotesize $\bullet$}; 
    \draw [red] (x5) node{\footnotesize $\bullet$};

    \end{tikzpicture}
    \end{center}
    
    \vspace{-0.5em}
    \caption{Local uniqueness of the prior-saturation point $x_e$.}
    \label{fig:bsb_sols}
\end{figure}

Assuming that the prior-saturation lift $z_e$ is locally unique, we can compute it with a set of equations
excerpt from the shooting equations but with some minor modifications.
Roughly speaking, the main idea is to consider the particular case where the initial condition is the prior-saturation point,
that is such that $x_0 = x_e$. In this case, we have $t_1 = t_2 = 0$ and 
$z_1 = z_2 = (x_e, p_0) = z_e$.
With these considerations in mind we introduce 
\[
    F_{\mathrm{ex}}(t_b, z_b) \coloneqq 
    \begin{pmatrix}
        H_g(\exp(-t_b \vv{H_+})(z_b)) \\
        H_{[f,g]}(\exp(-t_b \vv{H_+})(z_b)) \\
        H_+(z_b) + p^0 \\
        \pi(z_b) - x_f
    \end{pmatrix},
\]
with $F_{\mathrm{ex}} : \R^{5} \to \R^{5}$.
Note that the exponential mapping is here computed by backward integration. Hence, with the preceding notation, we have
$z_b = \exp( (t_f-t_2) \vv{H_+})(z_2) = \exp( t_f \vv{H_+})(z_e)$ and $t_b = t_f-t_2 = t_f$. At the end, the prior-saturation lift is
simply given by 
\[
    z_e = \exp(-t_b \vv{H_+})(z_b),
\]
for a couple $(t_b, z_b)$ solution of $F_{\mathrm{ex}} = 0$.

%
\paragraph*{Tangency property.}
We end this section with an introduction to the tangency property. 
Let us start with solutions of the form $\arc_- \arc_s \arc_+$,  
considering a smooth and local one-parameter family of initial conditions $x_0(\alpha)$, $\alpha \in (-\veps, \veps)$, $\veps > 0$, 
but assuming that for $\alpha = 0$, the optimal solution is of the form $\arc_- \arc_s \arc_+$ with $\arc_s$ reduced to a single
point, that is, $t_2(0)-t_1(0) = 0$, with $$y^*(0) \coloneqq (p_0(0), t_1(0), t_2(0), t_f(0), z_1(0), z_2(0)),$$ the solution to the associated
shooting equations, still denoted $S(y,\alpha) = 0$.
Assume also that for $\alpha > 0$, we are in the previous case, that is one has $t_2(\alpha)-t_1(\alpha) > 0$ with 
$
y^*(\alpha) \coloneqq (p_0(\alpha), t_1(\alpha), t_2(\alpha), t_f(\alpha), z_1(\alpha), z_2(\alpha))
$
the corresponding solution of
$S(\cdot, \alpha)=0$. The prior-saturation lift is thus given by $z_e = z_2(\alpha)$ for $\alpha \in [0,\veps)$.

The idea is now to consider the case where there is a bifurcation in the structure of the optimal trajectories when $\alpha=0$.
We thus assume that for $\alpha \in (-\veps,0)$,  the solutions are of the form $\arc_- \arc_+$ and we denote by $z_1(\alpha)$ the
switching point (in the cotangent bundle) between the two arcs. In this setting, there exists a switching locus in the optimal synthesis
denoted $\Sigma_- \cup \Sigma_0$, where 
\[
    \Sigma_- \coloneqq \{z_1(\alpha);~ \alpha \in (-\veps,0]\} \quad \text{and} \quad \Sigma_0 \coloneqq \{ z_1(0) = z_e \}.
\]

The aim of the next section is to prove that the semi-orbit $\Gamma_+$ of $\dot{z} = \vv{H_+}(z)$ starting from $z_e$ is tangent
to the switching curve $\Sigma_- \cup \Sigma_0$ at the prior-saturation lift $z_e$ in a general frame.
This is precisely the {\it{tangency property}} (see Fig.~\ref{fig:tangence}). 

\begin{figure}[ht!]

    \begin{center}
    \begin{tikzpicture}[scale=0.05]

        \coordinate (A) at (-10,-10);
        \coordinate (B) at (60,60);
        \coordinate (C) at (140,15);
    
        \coordinate (x1) at (30,-10);
        \coordinate (x2) at (95,-15);
        \coordinate (x3) at (25.5,31);
        \coordinate (x4) at (90,18.5);
        \coordinate (x5) at (75,50);
        \coordinate (x6) at (115,45);
        
        \draw [thin, black] (A) .. controls +(5,15) and +(-10,-5) .. (B) node[pos=0.8, left]{$\Sigma_\mathrm{SA}~$};
        
        \draw [thick, black] (x3) .. controls +(20,-10) and +(-15,0) .. (C) node[pos=0.85, above]{$\Sigma_-$};
    
        \draw [thin, blue] (x1) .. controls +(-4,10) and +(-2,-6) .. (x3) node[midway, sloped, rotate=180] {\scriptsize \ding{228}};
        \draw [thin, blue] (x2) .. controls +(-4,10) and +(-2,-6) .. (x4) node[midway, sloped, rotate=180] {\scriptsize \ding{228}};
        \draw [thick, red] (x3) .. controls +(30,-15) and +(-4,-6) .. (x5) node[midway, sloped] {\scriptsize \ding{228}} node[near end, left]{$\Gamma_+~$};
        \draw [thin, red] (x4) .. controls +(10,5) and +(-4,-6) .. (x6) node[midway, sloped] {\scriptsize \ding{228}};
        
        \draw [left] (x3) node{$z_{e}$};
        \draw [below] (x1) node{$z_0(0)$};
        \draw [below] (x2) node{$z_0(\alpha)$, $\alpha < 0$};
        \draw [below left] (x4) node{$z_1(\alpha)$};
        
        \draw [blue] (x1) node{\footnotesize $\bullet$}; 
        \draw [blue] (x2) node{\footnotesize $\bullet$}; 
        \draw [red] (x3) node{\footnotesize $\bullet$}; 
        \draw [red] (x4) node{\footnotesize $\bullet$};

    \end{tikzpicture}
    \end{center}
    
    \vspace{-0.5em}
    \caption{Illustration of the tangency property between $\Gamma_+$ and $\Sigma_- \cup \Sigma_0$ at the prior-saturation lift $z_e$.
    The \emph{singular locus} in the cotangent bundle is $\Sigma_\mathrm{SA} \coloneqq \{z\in \R^{2n}\; ; \;  H_g(z) = H_{[f,g]}(z) = 0\}$.}
    \label{fig:tangence}
\end{figure}
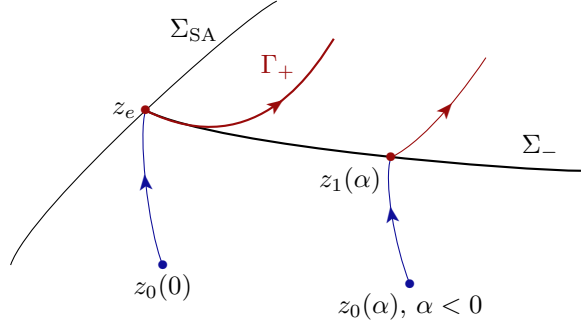

\subsection{Proof of the tangency property}

From a general point of view, we shall assume that the prior-saturation lift is given by solving a set of nonlinear equations 
$$
F(t_b, z_b, \lambda)=0
$$
with
\begin{equation}
    F(t_b, z_b, \lambda) \coloneqq
    \begin{pmatrix}
        H_{[f,g]}(\exp(-t_b \vv{H_+})(z_b)) \\
        G(t_b, z_b, \lambda)
    \end{pmatrix},
    \label{eq:F_prior_lift}
\end{equation}
where $\lambda \in \R^k$ is a vector of $k \in \N$ parameters, where $F$ is a function from $\R^{5+k}$ to 
$\R^{5+k}$ and where $G:\R^{5+k} \rightarrow \R^{4+k}$ is defined by
\begin{equation}
    G(t_b, z_b, \lambda) \coloneqq
    \begin{pmatrix}
        H_g(\exp(-t_b \vv{H_+})(z_b)) \\
        H_+(z_b) + p^0 \\
        \Psi(z_b, \lambda)
    \end{pmatrix},
    \label{eq:G_prior_lift}
\end{equation}
with $\Psi:\R^{4+k} \rightarrow \R^{2+k}$ a given function and $p^0 = -1$ considering the normal case.
We assume that all the functions $F$, $G$ and $\Psi$ are smooth.
It is important to notice that the mapping $\Psi$ does not depend on $t_b$ and that we can replace $H_+$ by $H_-$ without any loss of generality.
In the previous example from section \ref{sec:prior_lift_intro}, we have (with a slight abuse of notation) $\Psi(z_b) = \pi(z_b) - x_f$
which corresponds to the simplest case where there are no transversality conditions and no additional parameters, that is $k=0$.
For a more complex structure of the form $\arc_- \arc_s \arc_+ \arc_-$,
the parameter $\lambda$ would be the last switching time between the $\arc_+$ and $\arc_-$ arcs. In this case, $\Psi$ would contain
the additional switching condition $H_g=0$ at this time.

Let $(t_b^*, z_b^*, \lambda^*) \in \R^{5+k}$ be a solution to the equation $F=0$ and define 
\begin{equation}
    z_e \coloneqq \exp(-t_b^* \vv{H_+})(z_b^*) \in \Sigma_\mathrm{SA} := \{z\in \R^{2n}\; ; \;  H_g(z) = H_{[f,g]}(z) = 0\}.
    \label{eq:ze_prior_lift}
\end{equation}
We introduce the following assumptions at the point $z_e$. 
\begin{assumption}{\label{assumption:us_non_saturating}}
    We have $H_{[g,[f,g]]}(z_e) \ne 0$ and $u_s(z_e)<1$ with $u_s$ the singular control given by \eqref{eq:singular_control_z}.
\end{assumption}
\begin{assumption}{\label{assumption:Gprime_invertible}}
    The matrix 
    \[
        \begin{bmatrix}
            \displaystyle \frac{\partial G}{\partial z_b}(t_b^*, z_b^*, \lambda^*) & 
            \displaystyle \frac{\partial G}{\partial \lambda}(t_b^*, z_b^*, \lambda^*)
        \end{bmatrix}
        \in \GL_{4+k}(\R),
    \]
    {\it{i.e.}}, it is invertible in $\R^{(4+k)\times(4+k)}$. 
\end{assumption}

\begin{remark}
\label{rmk:assumptions_lift}
Assumption \ref{assumption:us_non_saturating} is related to the prior-saturation phenomenon while in combination with 
Assumption \ref{assumption:Gprime_invertible}, it is related to the well-posedness of the shooting system $F=0$.
Besides, the point $z_e$ is locally unique under these assumptions, according to the following result.
\end{remark}

\begin{proposition}
    \label{prop:prior_lift_unique}
    Suppose that Assumptions \ref{assumption:us_non_saturating} and \ref{assumption:Gprime_invertible} hold true. Then, 
    \[
        F'(t_b^*, z_b^*, \lambda^*) \in \GL_{5+k}(\R).
    \]
    Moreover $z_e$ (defined by \eqref{eq:ze_prior_lift}) is locally unique.
\end{proposition}
\begin{proof}
    The Jacobian of the mapping $F$ at the point $(t_b^*, z_b^*, \lambda^*)$ is given by:
    \[
        F'(t_b^*, z_b^*, \lambda^*) =
        \begin{bmatrix}
            -a & * & * \\
            -b & \displaystyle \frac{\partial G}{\partial z_b}(t_b^*, z_b^*, \lambda^*) &
            \displaystyle \frac{\partial G}{\partial \lambda}(t_b^*, z_b^*, \lambda^*)
        \end{bmatrix},
    \]
    where $a \coloneqq H_{[f,[f,g]]}(z_e)+H_{[g,[f,g]]}(z_e)$ and $b \coloneqq (H_{[f,g]}(z_e), 0, 0)$.  
    Then, we have $F'(t_b^*, z_b^*, \lambda^*) \in \GL_{5+k}(\R)$,
    noticing that $b=0$ since $F(t_b^*, z_b^*, \lambda^*) = 0$, that $a \ne 0$ since
    $u_s(z_e)<1$ and that $H_{[g,[f,g]]}(z_e)\ne 0$ by Assumption 
    \ref{assumption:us_non_saturating}. By the inverse function theorem $z_e$ is thus locally
    unique. 
\end{proof}

\begin{lemma}{\label{lmm:switching_curve}}
    Suppose that Assumption \ref{assumption:Gprime_invertible} holds true. Then, there exists $\eps>0$ and a $C^1$-map 
    $t_b \mapsto \sigma(t_b) \coloneqq (z_b(t_b), \lambda(t_b)) \in \R^{4+k}$ defined over  $I_\eps\coloneqq(t_b^*-\veps,t_b^*+\veps)$, that satisfies 
    \begin{equation}{\label{IFT1}}
    \forall t_b\in I_\eps, \quad G(t_b,\sigma(t_b))=0. 
    \end{equation}
    In addition, one has $\sigma(t_b^*) = (z_b^*, \lambda^*)$ and $\sigma'(t_b^*) = 0_{\R^{4+k}}$. 
\end{lemma}
\begin{proof}
    The existence of $\sigma$ follows from the implicit function theorem applied to the mapping $G$ at $(t_b^*, z_b^*, \lambda^*)$ which also gives \eqref{IFT1}. The derivative of $\sigma$ is then obtained from \eqref{IFT1}:
    \[
        \sigma'(t_b) = -
        \begin{bmatrix}
            \displaystyle \frac{\partial G}{\partial z_b}[t_b] & 
            \displaystyle \frac{\partial G}{\partial \lambda}[t_b]
        \end{bmatrix}^{-1}
        \cdot
        \frac{\partial G}{\partial t_b}[t_b], \quad t_b \in I_\eps, 
    \]
    where $[t_b]$ stands for $(t_b, \sigma(t_b))$. Since 
    \[
        \frac{\partial G}{\partial t_b}[t_b^*]= (H_{[f,g]}(z_e), 0_{\R^{3+k}}) = 0_{\R^{4+k}},
    \]
    the result follows.
\end{proof}

    Let us introduce the mapping $\vphi(t_b) \coloneqq \exp(-t_b \vv{H_+})(z_b(t_b))$ for $t_b \in I_\eps$ and define 
    \begin{equation}
        \Sigma \coloneqq \{ \vphi(t_b) \; ; \; t_b \in I_\veps \}.
        \label{eq:switching_curve}
    \end{equation}

\begin{remark}
    The curve $\Sigma$ is a switching curve in the contangent bundle since one has $H_g(\vphi(t_b))=0$ by definition of $G$.
    However, this switching curve is not necessarily optimal, that is, the optimal synthesis, with respect to the initial condition,
    may not contain $\Sigma$. 
    Let us stratify $\Sigma$ according to $\Sigma = \Sigma_- \cup \Sigma_0 \cup \Sigma_+$, with
    \[
        \begin{array}{l}
            \Sigma_- \coloneqq \{\vphi(t_b) \; ; \; t_b \in (t_b^*-\veps,t_b^*)\}, \vspace{0.1cm}\\
            \Sigma_0 \coloneqq \{\vphi(t_b^*)\} = \{z_e\}, \vspace{0.1cm}\\
            \Sigma_+ \coloneqq \{\vphi(t_b) \; ; \; t_b \in (t_b^*,t_b^*+\veps)\}.
        \end{array}
    \]
    A typical situation is when $\Sigma_- \cup \Sigma_0$
    is contained in the optimal synthesis while $\Sigma_+$ is not optimal for local and/or global optimality reasons.
    See the end of Section \ref{sec:prior_lift_intro} for an example of this typical situation.
\end{remark}

Our first main result is given by Proposition \ref{prior-vs-init} which states the existence of a prior-saturation point $x_e$ in the state
space under Assumption \ref{main-hyp}. Our second main result is the following.
\begin{theorem}{\label{thm-tangent}}
    Suppose the existence of a triple $(t_b^*, z_b^*, \lambda^*) \in \R^{5+k}$ such that $F(t_b^*, z_b^*, \lambda^*)= 0$, 
    with $F$ defined by \eqref{eq:F_prior_lift} and set $z_e \coloneqq \exp(-t_b^* \vv{H_+})(z_b^*)$.
    Suppose also that Assumption \ref{assumption:Gprime_invertible} holds true.
    Then, the switching curve $\Sigma$ given by \eqref{eq:switching_curve} is tangent at $z_e$ to the forward
    semi-orbit $\Gamma_+$ of $\dot{z} = \vv{H_+}(z)$ starting from $z_e$.
\end{theorem}
\begin{proof}
    From Assumption \ref{assumption:Gprime_invertible} and by lemma \ref{lmm:switching_curve}, one can define
    the switching curve $\Sigma$ by \eqref{eq:switching_curve}.
    To prove the tangency property, we have to show that $\vphi'(t_b^*)$ is collinear to $\vv{H_+}(z_e)$. For any $t_b \in I_\eps$,
    we have
    \[
        \vphi'(t_b) = -\vv{H_+}(\vphi(t_b)) + \exp(-t_b \vv{H_+})'(z_b(t_b)) \cdot z_b'(t_b).
    \]
    %
    By lemma \ref{lmm:switching_curve}, one has $\sigma'(t_b^*)=0$ thus $z_b'(t_b^*)=0$ and we get
    $\vphi'(t_b^*) = -\vv{H_+}(\vphi(t_b^*)) =-\vv{H_+}(z_e)$, which concludes the proof.
\end{proof}
\begin{remark}
    It is worth to mention that the tangency property is proved in the cotangent bundle, and thus it is also true in the state space at a prior saturation point (under the assumptions of Proposition \ref{existence-presat}).
\end{remark}

Setting $\xi(z) \coloneqq (H_g(z), H_{[f,g]}(z))$ the singular locus $\Sigma_\mathrm{SA}$ can be written $\Sigma_\mathrm{SA} = \xi^{-1}(\{0_{\R^2}\})$,
and we have the following relation between the singular locus and the switching curve.
\begin{corollary}
    \label{cor:tansversality_property}
    Suppose that $\xi$ is a submersion at $z_e$ and
    that Assumptions \ref{assumption:us_non_saturating} and \ref{assumption:Gprime_invertible} hold true.  
    Then the switching curve $\Sigma$ is transverse to the singular locus $\Sigma_\mathrm{SA}$ at $z_e$.
\end{corollary}
\begin{proof}
    Since $\xi$ is a submersion at $z_e$, the singular locus  $\Sigma_\mathrm{SA}$ is locally a regular submanifold of codimension two
    near $z_e$. Its tangent space at $z_e$ is given by the kernel of the matrix $\xi'(z_e)$. But,
    \[
        \begin{aligned}
              \xi'(z_e)\, \vphi'(t_b^*) &= -\xi'(z_e) \, \vv{H_+}(z_e) \quad \text{(see the proof of Theorem \ref{thm-tangent})} \\ 
            &=- \begin{pmatrix} 
             \partial_x H_g(z_e) & \partial_p H_g(z_e)\\
             \partial_x H_{[f,g]}(z_e) & \partial_p H_{[f,g]}(z_e)
            \end{pmatrix}
            \begin{pmatrix}
            \partial_p H_+(z_e) \\ -\partial_x H_+(z_e)
            \end{pmatrix}\\
            &= - \begin{pmatrix} H_{[f,g]}(z_e) \\ H_{[f,[f,g]]}(z_e)+H_{[g,[f,g]]}(z_e) \end{pmatrix} \\
            &\ne 0_{\R^2}  \quad \text{(by Assumption \ref{assumption:us_non_saturating})},
        \end{aligned}
    \]
    recalling that $\vphi$ is given from lemma \ref{lmm:switching_curve} by Assumption
    \ref{assumption:Gprime_invertible}.
\end{proof}

\begin{remark}
    From Theorem \ref{thm-tangent}, the tangency property holds even if the singular control at $z_e$ is saturating.
    The main reason of the tangency property comes from the fact that $z_e$ belongs to the singular locus $\Sigma_\mathrm{SA}$.
    However, if the singular control at $z_e$ is not saturating, for instance if $z_e$ is a prior-saturation lift, then the switching
    curve $\Sigma$ is transverse to the singular locus $\Sigma_\mathrm{SA}$ at $z_e$ according to Corollary
    \ref{cor:tansversality_property}.
\end{remark}

\section{Illustration of the prior-saturation phenomenon}{\label{sec-numer}}
The aim of this section is to develop two examples arising in the field of bioprocesses  and magnetic resonance imaging respectively, that will highlight the various concepts introduced in Sections \ref{sec-prior}-\ref{sec-tangent}. 
For the related minimal time problems, we shall also briefly discuss the corresponding optimal syntheses that exhibit prior-saturation points and bridges. 
\subsection{The fed-batch model}{\label{fb-sec}} 
A bioreactor operated in fed-batch is described by the controlled dynamics (see \cite{M99}):
\be{\label{eq-fed-batch2}}
\left\{
\begin{array}{cl}
\dot{s}&=-\mu(s)\left(\frac{M}{v}+s_{in}-s\right)+\frac{Q_{max}(1+u)}{2v}(s_{in}-s), \vspace{0.1cm}\\
\dot{v}&=\frac{Q_{max}}{2}(1+ u),
\end{array}
\right.
\ee 
where $s_{in}$ and $s$ denote respectively the input substrate and substrate concentrations, and $v$ is 
the volume of the reactor\footnote{In contrast with the previous sections in which state variables are $(x_1,x_2)$, we chose to adopt the notation $(s,v)$ that is  commonly used in the bioprocesses literature for fed-batch operations.}. 
The parameter $Q_{max}>0$ is  the maximal speed of the input pump (chosen large enough) 
so that $\frac{Q_{max}}{2}(1+u)$ represents the input flow rate, $u(\cdot)$ being the control variable with values in $[-1,1]$.  
Finally, $M\in \R$ depends on the initial value of micro-organism concentration\footnote{Micro-organism concentration $X>0$ can be expressed as a simple function of $(s,v)$, namely $X=M/v+s_{in}-s$, thus \eqref{eq-fed-batch2} is enough to describe a bioreactor operated in fed-batch mode.}. 
As in many engineering applications (see, {\it{e.g.}}, \cite{BernardAM2}), the kinetics $\mu$ of the reaction is of Haldane type, {\it{i.e.}}, $$
\mu(s)\coloneqq\frac{\mu_h s}{K+s+\frac{s^2}{K_I}},
$$  
with a unique maximum $s^*=\sqrt{KK_I}$ (parameters $\mu_h$, $K$, $K_I$ are positive). We suppose hereafter that $s^*$ belongs to $(0,s_{in})$. This hypothesis is not restrictive since from an applicative point of view, $s_{in}>0$ represents the input substrate concentration (to be treated). Hence, it can be expected that this value is greater than $s^*$, 
see, {\it{e.g.}}, \cite{BernardAM2}. 

This type of growth function models inhibition by substrate (microbial growth is limited when $s$ is too large w.r.t.~$s^*$). 
It is worth mentioning that the set $\mathcal{D}:=(0,s_{in}]\times \R_+^*$ is invariant by \eqref{eq-fed-batch2}. 
For waste water treatment purpose, the problem of interest is:
\be{\label{OCP-FB}}
\inf_{u \in \mathcal{U}} T_u \quad \mathrm{s.t.} \;\; (s(T_u),v(T_u))\in \mathcal{T}, 
\ee
where $\mathcal{T}\coloneqq(0,s_{ref}]\times \{v_{max}\}$ is the target set, $s_{ref}\ll s_{in}$ is a given threshold, and
$v_{max}>0$ denotes the maximal volume of the bioreactor. From a practical point of view, the goal is to treat a volume $v_{max}$ 
of wasted water in minimal time. For more details about this system, we refer to \cite{M99,BMM2}. 

It appears that Problem \eqref{OCP-FB} may exhibit a saturation phenomenon. 
Indeed, by using the PMP, we can check that there is a 
singular locus that is the line segment 
$$\Delta_{SA}= \{s^*\}\times(0,v_{max}],$$ 
and that the singular control can be expressed in feedback form as
$$u_s[v]\coloneqq
\frac{\mu(s^*)\left[M+v(s_{in}-s^*)\right]}{(s_{in}-s^*)Q_{max}}-1,
$$
(writing $\dot{s}=0$ along $s=s^*$). 
It follows that there exists a unique saturation point $$
x_{sat}=(s^*,v^*),
$$
with $v^*\coloneqq\frac{2Q_{max}}{\mu(s^*)}-\frac{M}{s_{in}-s^*}$ and $u_s[v^*]=1$ 
if 
the following condition is fulfilled 
\be{\label{hyp-saturation-bio}}
0<v^*<v_{max}.
\ee
This typically happens when the volume of water to be trated, $v_{max}$, is large w.r.t.~other parameters of the model (like in large scale waste water treatement industries), see \cite{BMM2}. Next, we suppose that \eqref{hyp-saturation-bio} holds true. 

At this step, we wish to know if prior-saturation occurs (according to Propositions \ref{existence-presat} and \ref{prior-vs-init}). 
Doing so, let us check Assumption \ref{main-hyp}. One gets
$$
\delta_0(s,v)=-\mu(s)(M/v+s_{in}-s)Q_{max}/2=-\mu(s)XQ_{max}/2 <0,
$$
hence $\Delta_0\cap \mathcal{D}=\emptyset$ and $\delta_0<0$ in $\mathcal{D}$. 
Now, the singular arc is of turnpike type and Legendre-Clebsch's optimality condition holds true because $\mu$ has a unique maximum for $s=s^*$, see \cite{BayenGM}, or a clock form argumentation in \cite{M99}. In addition, observe that, in the $(s,v)$-plane, trajectories of \eqref{eq-fed-batch2} with $u=-1$ are horizontal, hence, every arc with $u=-1$ and starting at a volume value $v_0<v_{max}$ never reaches the target set $\mathcal{T}$. Finally, 
$\mathcal{T}$ is reachable from $\mathcal{D}$ taking the control $u=+1$ until reaching $v=v_{max}$ and then $u=-1$ until reaching $s_{ref}$.     

Second, let us verify the hypotheses of Proposition \ref{prior-vs-init}. Doing so, 
let $v\mapsto \hat s(v)$ be the unique solution to the Cauchy problem
$$
\frac{ds}{dv}=\left(-\frac{\mu(s)}{Q_{max}}\left[\frac{M}{v}+s_{in}-s\right]+
\frac{s_{in}-s}{v}
\right), \quad s(v_{max})=s^*,
$$
(the solution of \eqref{eq-fed-batch2} with $u=1$ backward in time from $(s^*,v_{max})$). 
From \cite{BMM2}, if there exists $v_*\in (0,v^*)$ such that $\hat s(v_*)=s^*$, then optimal paths starting at a volume value sufficiently small necessarily contain a singular arc (this actually follows using the PMP). Now, by using Cauchy-Lipschitz's Theorem, the existence of $v_*$ is easy to verify when $M=0$, and thus, it is also verified 
for small values of the parameter $M$ (by a continuity argumentation). To pursue our analysis, we suppose next the existence of $v_* \in (0,v^*)$. 
We are then in a position to apply Propositions \ref{existence-presat} and \ref{prior-vs-init}. 
It follows that there is a unique volume value $v_e\in (0,v^*)$ such that any singular arc starting at a volume value $v_0<v_e$ will be optimal only until $v_e$. In addition, combining this result with a study of extremals using the PMP, we obtain that
\begin{itemize}
\item if the initial condition is $(s^*,v_0)$ with $v_0<v_e$, then the optimal path is of the form $\sigma_s\sigma_+^b\sigma_-$ (see below for the definition of $\sigma_+^b$);  
\item if the initial condition is $(s^*,v_0)$ with $v_0\geq v_e$, then the optimal path is of the form $\sigma_+\sigma_-$ ; 
\item for any initial condition $(s_{in},v_0)$ with $v_e\leq v_0<v_{max}$, the optimal path is of the form 
$\sigma_- \sigma_+ \sigma_-$ where the first switching time appears on a switching curve 
emanating from $(s^*,v_e)$. 
\end{itemize}
To determine the prior-saturation point $x_e=(s^*,v_e)$ numerically, we proceed as in Section \ref{sec-tangent}. 
For this application model, it is convenient to introduce an {\it{extended target set}} as   $\overline{\mathcal{T}}=(0,s_{in}] \times \{v_{max}\}$ 
(observe that for initial conditions on $\overline{\mathcal{T}}$, 
optimal paths are $\sigma_-$ arcs).  
In this context, a {\it{bridge}} is defined as an arc $\sigma_+$ (denoted by $\sigma_+^b$) on $[0,t_b]$ such that 
\[
    \phi(0)=\dot{\phi}(0)=\phi(t_b)=0 \quad \mathrm{and} \quad v(t_b)=v_{max},
\]
where $\phi$ is the switching function defined by \eqref{eq:switching_function} and $t_b$ is the time to 
steer $x_e$ at time $0$ to the extended target set $\overline{\mathcal{T}}$ with $u=+1$. To compute $x_e$, we need to compute the extremities of the bridge
together with its length. Denoting by $t_b^*$ the length of the bridge and by 
$z_b^*$ its extremity in the cotangent bundle whose projection on the state space belongs to $\overline{\mathcal{T}}$, the point
$(t_b^*, z_b^*)$ is then a solution of the equation $F_\mathrm{bio} = 0$ with
\begin{equation}
    F_\mathrm{bio}(t_b, z_b) \coloneqq
    \begin{pmatrix}
        H_{[f,g]}(\exp(-t_b \vv{H_+})(z_b)) \\
        H_g(\exp(-t_b \vv{H_+})(z_b)) \\
        H_+(z_b) + p^0 \\
        H_g(z_b) \\
        v_b-v_{max} \\
    \end{pmatrix},
    \label{eq:F_prior_lift_bio}
\end{equation}
where $(s_b,v_b)$ is the projection of $z_b$ on the state space and vector fields $f,g$ are given by \eqref{eq-fed-batch2}.  
From Theorem \ref{thm-tangent}, the bridge is then tangent to the switching curve at $x_e$ (the projection of $\Sigma$ given by \eqref{eq:switching_curve} onto the state space). To conclude this part, let us comment  Fig.~\ref{fig:synthese_bio_1} on which the optimal synthesis is plotted in a neighborhood of the prior-saturation point:  
\begin{itemize}
\item In black, the switching curve $\Sigma^\pi$ emanates from the prior-saturation point. It is computed using the shooting functions $F_{bio}$. 
\item The synthesis is such that trajectories are horizontal ($u=-1$) until reaching $\Delta_{SA}$ or the switching curve. For initial conditions with a substrate concentration less than $s^*$ and $v_0\geq v_e$, then $u=1$ is optimal until reaching $\overline{\mathcal{T}}$. 
\end{itemize}

\begin{figure}[ht!]
\centering
\begin{tikzpicture}[scale=0.05]

    \def\bordG{-10}
    \def\bordD{220}
    \def\bordB{-10}
    \def\bordH{115}

    \coordinate (Nord) at (0,\bordH);
    \coordinate (Sud) at (0,\bordB);
    \coordinate (Est) at (\bordD,0);
    \coordinate (Ouest) at (\bordG,0);

    \draw [->,thin] (Ouest) -- (Est)  node[below] {$s$};
    \draw [->,thin] (Sud)   -- (Nord) node[left]  {$v$};
    
    \def\csin{200}
    \def\csstar{100}
    \def\csref{20}
    \def\csA{40}
    \def\csB{65}
    \def\csC{85}
    
    \def\cvA{10}
    \def\cve{35}
    \def\cvB{45}
    \def\cvsat{52}
    \def\cvC{70}
    \def\cvmax{100}
    
    \draw [densely dashed, thin] (\csin,0) -- (\csin,\bordH); \draw [below] (\csin,0) node{$s_{in}$};
    
    \draw [densely dashed, thin] (0, \cvmax) -- (\bordD, \cvmax); \draw [left] (0, \cvmax) node{$v_{max}$};
    
    \draw [thin] (\csstar,0) -- (\csstar,\bordH); 
    \draw [below] (\csstar,0) node{$s^*$}; \draw [below right] (\csstar,\bordH) node{$\Delta_{SA}$};
    
    \draw [thin] (0,\cvmax) -- (\csin,\cvmax) node[near end, above] {$\overline{\mathcal{T}}$};
    
    \draw [very thick] (0,\cvmax) -- (\csref,\cvmax) node[midway, above] {${\mathcal{T}}$};
    \def\ytick{5em}
    \draw [very thick] ([yshift=-\ytick] \csref,\cvmax) -- ([yshift=\ytick] \csref,\cvmax);
    \draw [very thick] ([yshift=-\ytick] 0,\cvmax) -- ([yshift=\ytick] 0,\cvmax);

    \draw [left] (\csstar, \cve) node{$x_{e}$};
    \draw [left] (\csstar, \cvsat) node{$x_{sat}$};

    \draw [thick, green] (\csstar,\cvA) -- (\csstar,\cve) node[midway, sloped] {\scriptsize \ding{228}}
    node[midway, left]{$\textcolor{green}{\arc_s}$};
    
    \draw [thick, blue] (\csin,\cvA) -- (\csstar,\cvA) node[midway, sloped, rotate=180] {\scriptsize \ding{228}}
    node[midway, below]{$\textcolor{blue}{\arc_-}$};    
    
    \draw [thick, blue] (\csin,\cve) -- (\csstar,\cve) node[midway, sloped, rotate=180] {\scriptsize \ding{228}}
    ; 
    
    \def\pentex{18}
    \def\pentey{14}
    \def\csD{120}
    \draw [thick] (\csstar,\cve) .. controls +(\pentex,\pentey) and +(0,-10) .. (\csD,\cvC) node[midway, right] {$\Sigma^\pi$};
    \draw [thick] (\csD,\cvC) .. controls +(0,10) and +(15,-6) .. (\csstar,\cvmax);
    
    \draw [thick, red] (\csstar,\cve) .. controls +(2*\pentex,2*\pentey) and +(20,-5) .. (\csB,\cvmax)
    node[pos=0.6, sloped, rotate=180] {\scriptsize \ding{228}} node[pos=0.65, below] {$\textcolor{red}{\arc_+^b}~$};
    
    \draw [thick, red] (\csstar,\cvB) .. controls +(\pentex,\pentey) and +(20,-5) .. (\csA,\cvmax)
    node[pos=0.6, sloped, rotate=180] {\scriptsize \ding{228}}; 
    
    \draw [thick, red] (\csD,\cvC) .. controls +(-5,10) and +(20,-5) .. (\csC,\cvmax)
    node[pos=0.4, sloped, rotate=180] {\scriptsize \ding{228}}; 
    
    \draw [thick, blue] (\csin,\cvC) -- (\csD,\cvC) node[midway, sloped, rotate=180] {\scriptsize \ding{228}}
    ; 
    
    \draw [thick, blue] (\csC,\cvmax) -- (\csB,\cvmax) node[midway, sloped, rotate=180] {\scriptsize \ding{228}}
    ; 
    \draw [thick, blue] (\csB,\cvmax) -- (\csA,\cvmax) node[midway, sloped, rotate=180] {\scriptsize \ding{228}}
    node[midway, above]{$\textcolor{blue}{\arc_-}$};
    \draw [thick, blue] (\csA,\cvmax) -- (\csref,\cvmax) node[midway, sloped, rotate=180] {\scriptsize \ding{228}}
    ; 
    
    \draw [green] (\csstar, \cvA) node{\footnotesize $\bullet$};
    \draw [red] (\csstar, \cve) node{\footnotesize $\bullet$};
    \draw [black] (\csstar, \cvsat) node{\footnotesize $\bullet$};
    \draw [red] (\csstar, \cvB) node{\footnotesize $\bullet$};
    \draw [blue] (\csin, \cvA) node{\footnotesize $\bullet$};
    \draw [blue] (\csin, \cve) node{\footnotesize $\bullet$};
    \draw [blue] (\csin, \cvC) node{\footnotesize $\bullet$};
    \draw [blue] (\csA, \cvmax) node{\footnotesize $\bullet$};
    \draw [blue] (\csB, \cvmax) node{\footnotesize $\bullet$};
    \draw [blue] (\csC, \cvmax) node{\footnotesize $\bullet$};  
    \draw [red] (\csD, \cvC) node{\footnotesize $\bullet$};    
    
\end{tikzpicture}
\vspace{-1em}
\caption{Minimal time synthesis for \eqref{OCP-FB}: the target set $\mathcal{T}=(0,s_{ref}]\times \{v_{max}\}$ is in black (left).
The switching curve $\Sigma^\pi$ (in black) is tangent to the bridge $\sigma_+^b$ (in red) at $x_e$. Arcs with $u=+1$ (resp.~$u=-1$) are depicted in red (resp.~in blue).}
\label{fig:synthese_bio_1}
\end{figure}
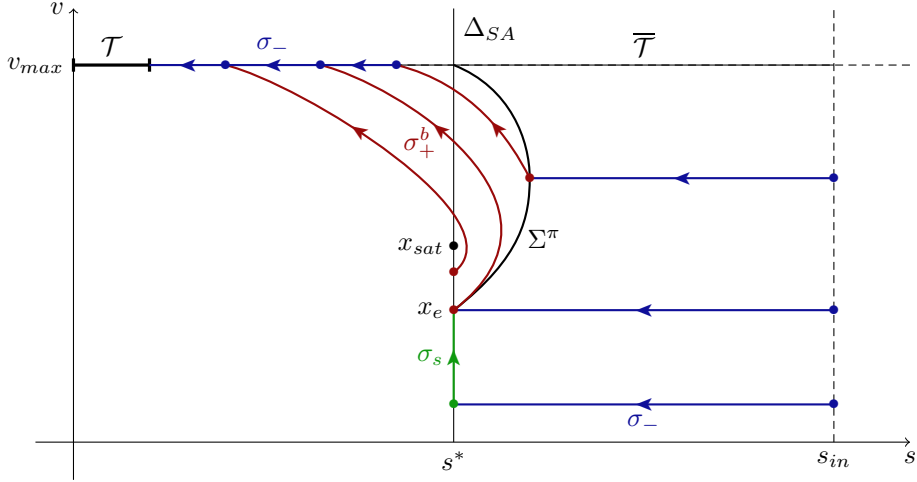

\subsection{The MRI model}{\label{mri-sec}} 

In Nuclear Magnetic Resonance (NMR) \emph{saturating} one chemical species consists in driving the magnetization
vector representing the state to zero. In Magnetic Resonance Imaging (MRI) a challenging problem is to maximize
the \emph{constrast} between two observed species (for instance, healthy tissues and tumors) saturating one species.
For the model, we consider an ensemble of spin-1/2 particles, excited by a radio-frequency (RF) field which 
is ideally assumed homogeneous, each spin of this ensemble being described by its magnetization vector whose
dynamics is governed in a specific rotating frame, after some normalizations and considering the 2-dimensional
case, by the
\emph{Bloch equation} \cite{Levitt:2008}:
\begin{equation}
    \label{eq:Bloch_Equation}
    \left\{ 
    \begin{aligned}
        \dot{x}_1 &= -\Gamma x_1 - u\, x_2, \\
        \dot{x}_2 &= \gamma (1-x_2) + u\, x_1,
    \end{aligned}
    \right.
\end{equation}
where $x\coloneqq(x_1, x_2)$ is the normalized magnetization vector, where $(\gamma, \Gamma)$ is a couple of parameters
satisfying the physical constraint $0 < \gamma \le 2\Gamma$ and depending on the longitudinal and transversal relaxation constants
specific to the observed species, and where $u$ is the RF-field which plays the role of the control.
The time-minimal problem of interest here
is the following:
\begin{equation}
    \label{ocp:MRI}
    \inf_{u \in \mathcal{U}} T_u \quad \mathrm{s.t.} \quad x_u(T_u, x_0) = O \coloneqq (0, 0), 
\end{equation}
where the initial condition $x_0$ belongs to the set $B\coloneqq\{ (x_1,x_2) \in \R^2~;~ x_1^2+x_2^2 \le 1 \}$ 
called the \emph{Bloch ball} and where $x_u(\cdot,x_0)$
is the unique solution of \eqref{eq:Bloch_Equation} such that $x_u(0, x_0) = x_0$.

\begin{remark}
    The problem of saturation in MRI is the problem \eqref{ocp:MRI} with $x_0 = N$, where $N \coloneqq (0,1)$
    is the North pole of the Bloch ball. We refer to {\rm{\cite{cots3, cots2}}} for more details about the saturation and
    contrast problems in MRI.
    In {\rm{\cite{cots3}}}, the following optimal synthesis is constructed:
    the authors give the optimal paths to go from $N$ to any reachable point of the Bloch ball. 
    Hence, the initial point is fixed to the North pole while the final point may be seen as a parameter.
    Here, we are interested in the converse problem, that is, the parameter is the initial condition and we want to steer the
    system to a given target, which is the origin $O$. 
    The common problem in these two cases is the problem of saturation where the initial condition is $N$ and where
    the target is $O$.
\end{remark}
 
In this MRI application \cite{cots3}, the singular locus has a singularity at the intersection of two lines.
Setting $\delta \coloneqq\gamma - \Gamma$, the singular locus is described by 
\[
    \Delta_{SA} = \Delta_{SA}^h \cup \Delta_{SA}^v,
\]
where $\Delta_{SA}^h \coloneqq \{x_2 = \gamma/(2\delta)\}$ is a horizontal line and where $\Delta_{SA}^v \coloneqq \{x_1=0\}$ 
is the vertical axis. On the vertical axis, the singular control is zero while on the
horizontal line, the singular control is given in feedback form by $$ u_s[x_1,x_2] \coloneqq \gamma(2\Gamma-\gamma)/(2\delta x_1).$$
Considering only the half space $x_1 \le 0$ of the Bloch ball (this is possible due to a discrete symmetry)
and restricting $(\gamma, \Gamma)$ to the interesting case $0 < 3\gamma \le 2\Gamma$ (in this case, the horizontal
line cuts the Bloch ball), there exists only one saturation point denoted by
$x_\mathrm{sat}\in \Delta_{SA}$. The point $x_\mathrm{sat}$ belongs to the set $\Delta_{SA}^h \cap \{x_1 < 0 \}$,
it satisfies $u_s(x_\mathrm{sat}) = 1$, and it is given by
\[
    x_\mathrm{sat} = \left( \frac{\gamma(2\Gamma-\gamma)}{2\delta}, \frac{\gamma}{2\delta} \right).
\]
Following \cite{cots3}, we introduce the concept of bridge.
An arc $\arc_+$ or $\arc_-$ with control $u=+1$ or $u=-1$, is called a \emph{bridge} on $[0,t_b]$
if its extremities correspond to non ordinary switching points, that is, if 
\[
    \phi(0)=\dot{\phi}(0)=\phi(t_b)=\dot{\phi}(t_b)=0,
\]
where $\phi$ is the switching function defined by \eqref{eq:switching_function}.
According to \cite{cots3}, there exists a bridge with $u=+1$ denoted $\sigma_+^b$ connecting $\Delta_{SA}^h$ and $\Delta_{SA}^v$.
We denote by $x_e \coloneqq (x_{e,1}, x_{e,2})$ the extremity of the bridge on the horizontal line $\Delta_{SA}^h$ and we
can now restrict  the analysis to the following situation.
We assume that the following conditions are satisfied by the couple of parameters $(\gamma, \Gamma)$
(see Fig.~3 of \cite{cots3} and the description that comes after for details):
\begin{enumerate}
    \item[$\bullet$] $x_e$ belongs to the Bloch ball $B$ (this implies in particular that $3\gamma \le 2\Gamma$),
    \item[$\bullet$] $0 < \gamma$ (this comes from the physical constraint), 
    \item[$\bullet$] $0 \le (2\Gamma^2-\gamma\Gamma+1) \exp((\alpha-\gamma)t_0) - 2 \delta$ (hence the origin $O$ is reachable by a Bang-Singular
    sequence from $x_\mathrm{sat}$ and so also from $x_e$),
\end{enumerate}
where $\alpha \coloneqq \delta/2$ and $t_0 \coloneqq \arctan( -{\beta}/{\alpha})/{\beta}$
with $\beta \coloneqq \sqrt{1-\alpha^2}$.
In this setting, for any initial condition $x_0 \coloneqq (x_{0,1}, x_{0,2}) \in \Delta_{SA}^h \cap B$ such that $x_{0,1} \le x_{e,1}$,
the optimal trajectory (see \cite{cots3}) is of the form $\arc_s \arc^b_+ \arc_0$, that is composed of a singular arc on $\Delta_{SA}^h$ followed
by the bridge with $u=+1$ and ending with a singular arc $\sigma_0$  along $\Delta_{SA}^v$ with $u=0$. The first singular arc reduces to a point
if $x_0 = x_e$. At $x_e$, the singular control is not saturating, so, in conclusion, the point $x_e$ is a prior-saturation point.

\begin{remark}
    In the MRI application, Assumption \ref{main-hyp} is not exactly satisfied since the collinearity set $\Delta_0$ is non-empty and plays a role
    in the optimal synthesis, such as the singularity of the singular locus at the intersection of the two lines.
    However, the singular arcs are turnpikes and Legendre-Clebsch optimality condition holds.
    Besides, there exists a prior-saturation point and so this case is more general than the fed-batch application.
    We will see hereinafter that the tangency property holds at the prior-saturation point and that the switching curve is transverse
    to the singular locus.
\end{remark}

We end this part by showing how to compute the prior-saturation point $x_e$ and by giving the optimal synthesis near $x_e$
for an initial condition on the horizontal singular line. To compute $x_e$, we need to compute the extremities of the bridge
together with its length. Denoting by $t_b^*$ the length of the bridge and by 
$z_b^*$ the extremity of the bridge in the cotangent bundle whose projection on the state space belongs to $\Delta_{SA}^v$, the point
$(t_b^*, z_b^*)$ is then a solution of the equation $F_\mathrm{mri} = 0$ with
\begin{equation}
    F_\mathrm{mri}(t_b, z_b) \coloneqq
    \begin{pmatrix}
        H_{[f,g]}(\exp(-t_b \vv{H_+})(z_b)) \\
        H_g(\exp(-t_b \vv{H_+})(z_b)) \\
        H_+(z_b) + p^0 \\
        H_{[f,g]}(z_b) \\
        H_g(z_b)
    \end{pmatrix},
    \label{eq:F_prior_lift_mri}
\end{equation}
where the vector fields $f$ and $g$ are given by \eqref{eq:Bloch_Equation} and where the Hamiltonians, the Hamiltonian lifts
and the Hamiltonian vector field are defined in Section \ref{sec:prior_lift_intro}.
We recognize here a function of the form \eqref{eq:F_prior_lift} without any additional parameter $\lambda$ and so, 
$z_e \coloneqq \exp(-t_b^* \vv{H_+})(z_b^*)$ is the prior-saturation lift such that $\pi(z_e) = x_e \in \Delta_{SA}^h$.
Finally, the optimal synthesis near $x_e$ is given on Fig.~\ref{fig:synthesis_mri}.
The optimal solution from the initial condition $x_0 \in \Delta_{SA}^h$ is of the form 
$\textcolor{black}{\arc_s} \textcolor{black}{\arc_+^b} \textcolor{black}{\arc_0}$. The red arc $\textcolor{black}{\arc_+^b}$
is the bridge starting from $x_e$, it is a part of the forward semi-orbit $\Gamma_+$ of $\dot{z} = \vv{H_+}(z)$ starting from $z_e$
projected into the state space. The black curve $\Sigma^\pi$ is the existing part in the optimal synthesis of the projection of the
switching curve $\Sigma$ defined by \eqref{eq:switching_curve}. According to the tangency property from Theorem \ref{thm-tangent},
the arc $\textcolor{black}{\arc_+^b}$ is tangent to $\Sigma^\pi$ at the prior-saturation point $x_e$.
Note also that the switching curve $\Sigma^\pi$ is transverse to the singular locus $\Delta_{SA}^h$ in accordance with Corollary 
\ref{cor:tansversality_property}.

\begin{remark}
    One can see from \eqref{eq:F_prior_lift_bio} and \eqref{eq:F_prior_lift_mri} that neither $F_\mathrm{bio}$ nor $F_\mathrm{mri}$ include any conditions about reaching the target.
    Indeed, to compute the prior-saturation lift we do not necessarily need the final
    conditions from the shooting functions. We only need a part of the shooting function which
    satisfies Assumptions \ref{assumption:us_non_saturating} and 
    \ref{assumption:Gprime_invertible}. In these two applications, we have by Proposition
    \ref{prop:prior_lift_unique} that the bridge is locally unique.
    Note that this gives us a new procedure to solve the shooting equations: we can first solve the reduced problem of computing the prior-saturation lift and then we can use this information to connect the initial condition to the prior-saturation lift and to connect
    the extremity of the bridge to the target.
    
\end{remark}

\begin{figure}[ht!]
\centering
\begin{tikzpicture}[scale=0.05]

    \def\zsing{-50}
    \def\pentex{20}
    \def\pentey{-25}

    \coordinate (O) at (0,0);
    \coordinate (N) at (0,100);
    \coordinate (W) at (-100,0);
    \coordinate (S) at (0,-100);
    \coordinate (A) at (-80,\zsing);
    \coordinate (L) at (-60,\zsing);
    \coordinate (B) at (-40,\zsing);
    \coordinate (D) at (0,-25);
    \coordinate (X) at (0,\zsing);
    \coordinate (Nord) at (0,30);
    \coordinate (Sud) at (0,-110);
    \coordinate (Est) at (30,0);
    \coordinate (Ouest) at (-110,0);

    \draw [->,thin] (Ouest) -- (Est)  node[below] {$x_1$};
    \draw [->,thin] (Sud)   -- (Nord) node[left]  {$x_2$};
    
    \draw (W) arc (180:270:100);
    \draw (W) arc (180:163:100);
    \draw (S) arc (270:290:100);
    
    \draw [densely dashed, thin] (-110,\zsing) -- (30,\zsing) node[below]{\small $\Delta_{SA}^h$};

    \draw [thick, green] (A) -- (L) node[midway, sloped] {\scriptsize \ding{228}} node[midway, below]{$\textcolor{green}{\arc_s}$};

    \draw [thick, green] (D) -- (O) node[midway, sloped] {\scriptsize \ding{228}} node[midway, right]{$\textcolor{green}{\arc_0}$};

    \draw [thick] (L) .. controls +(\pentex,\pentey) and +(-10,-7) .. (X) node[midway, below] {$\Sigma^\pi$};

    \draw [thick, red] (L) .. controls +(\pentex,\pentey) and +(-10,-20) .. (D)
    node[near end, sloped] {\scriptsize \ding{228}} node[near end, above] {$\textcolor{red}{\arc_+^b}$~~~};

    \draw (A) node[above]       {$x_0$} node[green]{\footnotesize $\bullet$};
    \draw (L) node[above]       {$x_e$} node[red]{\footnotesize $\bullet$};
    \draw (B) node[above]       {$x_\mathrm{sat}$} node[black]{\footnotesize $\bullet$};
    \draw (O) node[above right] {$O$} node{\footnotesize $\bullet$};
    \draw (D) node[green] {\footnotesize $\bullet$};

\end{tikzpicture}
\vspace{-1em}
\caption{Optimal synthesis near the prior-saturation point $x_e$ in the left part of the Bloch ball.}
\label{fig:synthesis_mri}
\end{figure}
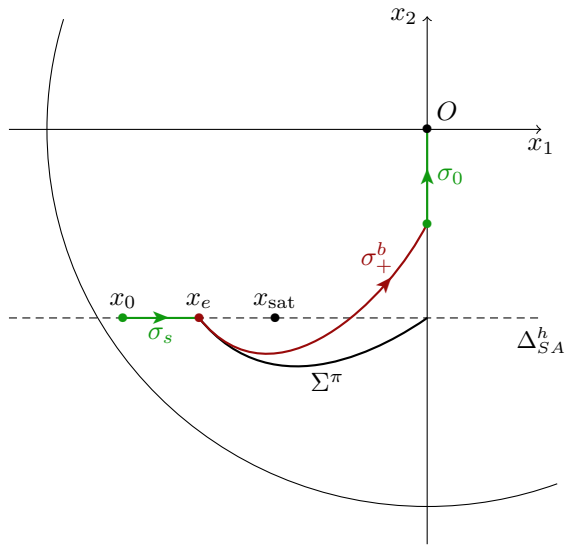

\section{Conclusion}
Even though the tangency property between the bridge and the switching curve provides useful informations on the minimum time synthesis 
when prior saturation occurs (typically, under assumptions of Proposition \ref{existence-presat}), it remains valid in a larger context (under the  hypotheses of Theorem \ref{thm-tangent}) and not only in the framework of saturation and prior-saturation of the singular control for affine-control systems in the plane. 
This property also appears in other settings such as in Lagrange control problems governed by one-dimensional systems, see, {\it{e.g.}}, \cite{filtration2017}. 
Future works could then investigate  prior-saturation phenomenon and the  tangency property in other frameworks or in dimension $n\geq 3$. 
\section*{Acknowledgments}
We are very grateful to E. Tr\'elat for helpful discussions and suggestions about the tangency property at the prior-saturation point. 


\begin{thebibliography}{10}

 \bibitem{BayenGM}
 {\sc T.~Bayen, P.~Gajardo, and F.~Mairet}, {\em Optimal synthesis for the
   minimum time control problems of fed-batch bioprocesses for growth functions
   with two maxima}, J. Optim. Theory Appl., 158 (2013), pp.~521--553.

 \bibitem{BHS}
 {\sc T.~Bayen, J.~Harmand, and M.~Sebbah}, {\em Time-optimal control of
   concentration changes in the chemostat with one single species}, Applied
   Mathematical Modelling, 50 (2017), pp.~257--278.

 \bibitem{BMM2}
 {\sc T.~Bayen, M.~Mazade, and F.~Mairet}, {\em Analysis of an optimal control
   problem connected to bioprocesses involving a saturated singular arc},
   Discrete Contin. Dyn. Syst. Ser. B, 20 (2015), pp.~39--58.

 \bibitem{BRS14}
 {\sc T.~Bayen, A.~Rapaport, and M.~Sebbah}, {\em Minimal time control of the
   two tanks gradostat model under a cascade input constraint}, SIAM J. Control
   Optim., 52 (2014), pp.~2568--2594. 

 \bibitem{BernardAM2}
 {\sc O.~Bernard, Z.~Hadj-Sadok, D.~Dochain, A.~Genovesi, and J.-P. Steyer},
   {\em Dynamical model development and parameter identification for an
   anaerobic wastewater treatment process}, Biotechnology and Bioengineering, 75
   (2001), pp.~424--438. 

 \bibitem{bonnard}
 {\sc B.~Bonnard and M.~Chyba}, {\em Singular trajectories and their role in
   control theory}, vol.~40 of Math\'ematiques \& Applications (Berlin)
   [Mathematics \& Applications], Springer-Verlag, Berlin, 2003.

 \bibitem{BD2020}
 {\sc B. Bonnard, J. Drouot}, {\em 
  Towards Geometric Time Minimal Control without Legendre Condition and with Multiple Singular Extremals for Chemical Networks}, 
 \url{https://arxiv.org/abs/2001.04126}


\bibitem{kupka}
{\sc B. Bonnard, and I. Kupka}, {\em Generic properties of singular trajectories}, Annales de l'Institut Henri Poincare (C) Non Linear Analysis. Vol. 14, 2, 1997, pp. 167--186. 

 \bibitem{cots2}
 {\sc B.~Bonnard, M.~Claeys, O.~Cots, and P.~Martinon}, {\em Geometric and
   numerical methods in the contrast imaging problem in nuclear magnetic
   resonance}, Acta Appl. Math., 135 (2015), pp.~5--45.

 \bibitem{cots3}
 {\sc B.~Bonnard, O.~Cots, J.~Rouot, and T.~Verron}, {\em Time minimal saturation
   of a pair of spins and application in magnetic resonance imaging}, Math.
   Control Related Fields,  (2019) 

 \bibitem{cots}
 {\sc B.~Bonnard, O.~Cots, S.~Glaser, M.~Lapert, D.~Sugny, and Y.~Zhang}, {\em
   Geometric optimal control of the contrast imaging problem in nuclear magnetic
   resonance}, IEEE Trans. Automat. Control, 57 (2012), pp.~1957--1969.

 \bibitem{BDM95}
 {\sc B.~Bonnard and J.~De~Morant}, {\em Toward a geometric theory in the
   time-minimal control of chemical batch reactors}, SIAM J. Control Optim., 33
   (1995), pp.~1279--1311. 

 \bibitem{bonnard95}
 {\sc B.~Bonnard and M.~Pelletier}, {\em Time minimal synthesis for planar
   systems in the neighborhood of a terminal manifold of codimension one}, J.
   Math. Systems Estim. Control, 5 (1995), p.~22.

 \bibitem{bosc}
 {\sc U.~Boscain and B.~Piccoli}, {\em Optimal syntheses for control systems on
   2-{D} manifolds}, vol.~43 of Math\'ematiques \& Applications (Berlin)
   [Mathematics \& Applications], Springer-Verlag, Berlin, 2004.

 \bibitem{Hermes}
 {\sc H.~Hermes and J.~LaSalle}, {\em Functional analysis and time optimal
   control}, Academic Press, New York-London, 1969.
 \newblock Mathematics in Science and Engineering, Vol. 56.

 \bibitem{filtration2017}
 {\sc N.~{Kalboussi}, A.~{Rapaport}, T.~{Bayen}, N.~B. {Amar}, F.~{Ellouze}, and
   J.~{Harmand}}, {\em Optimal control of membrane-filtration systems}, IEEE
   Transactions on Automatic Control, 64 (2019), pp.~2128--2134.

 \bibitem{ledje2}
 {\sc U.~Ledzewicz and H.~Sch\"attler}, {\em Antiangiogenic therapy in cancer
   treatment as an optimal control problem}, SIAM J. Control Optim., 46 (2007),
   pp.~1052--1079.

 \bibitem{Levitt:2008}
 {\sc M.~Levitt}, {\em Spin Dynamics: Basics of Nuclear Magnetic Resonance},
   Wiley, 2008. 

 \bibitem{M99}
 {\sc J.~Moreno}, {\em Optimal time control of bioreactors for the wastewater
   treatment}, Optimal Control Appl. Methods, 20 (1999), pp.~145--164.

 \bibitem{piccoli1}
 {\sc B.~Piccoli}, {\em Classification of generic singularities for the planar
   time-optimal synthesis}, SIAM J. Control Optim., 34 (1996), pp.~1914--1946.

 \bibitem{Pontry}
 {\sc L.~S. Pontryagin, V.~G. Boltyanskii, R.~V. Gamkrelidze, and E.~F.
   Mishchenko}, {\em The mathematical theory of optimal processes}, Translated
   by D. E. Brown, A Pergamon Press Book. The Macmillan Co., New York, 1964.

 \bibitem{landfill}
 {\sc A.~Rapaport, T.~Bayen, M.~Sebbah, A.~Donoso-Bravo, and A.~Torrico}, {\em
   Dynamical modeling and optimal control of landfills}, Math. Models Methods
   Appl. Sci., 26 (2016), pp.~901--929.

 \bibitem{schaettler}
 {\sc H.~Sch\"attler and M.~Jankovic}, {\em A synthesis of time-optimal controls
   in the presence of saturated singular arcs}, Forum Math., 5 (1993),
   pp.~203--241. 

 \bibitem{schaettler2}
 {\sc H.~Sch\"attler and U.~Ledzewicz}, {\em Geometric optimal control}, vol.~38
   of Interdisciplinary Applied Mathematics, Springer, New York, 2012.

 \bibitem{sussmann1}
 {\sc H.~J. Sussmann}, {\em Regular synthesis for time-optimal control of
   single-input real analytic systems in the plane}, SIAM J. Control Optim., 25
   (1987), pp.~1145--1162. 

 \bibitem{sussmann3}
 {\sc H.~J. Sussmann}, {\em The structure of time-optimal trajectories for
   single-input systems in the plane: the {$C^\infty$} nonsingular case}, SIAM
   J. Control Optim., 25 (1987), pp.~433--465.

 \bibitem{sussmann2}
 {\sc H.~J. Sussmann}, {\em The structure of time-optimal trajectories for
   single-input systems in the plane: the general real analytic case}, SIAM J.
   Control Optim., 25 (1987), pp.~868--904.

 \bibitem{vinter}
 {\sc R.~Vinter}, {\em Optimal control}, Systems \& Control: Foundations \&
   Applications, Birkh\"auser Boston, Inc., Boston, MA, 2000.

 \end{thebibliography}
 
\end{document}